\documentclass[12pt]{article}

\usepackage{geometry,mathrsfs}
\usepackage{amsbsy,amsmath,amsthm,amssymb,graphicx,subcaption}
\usepackage{verbatim}
\usepackage{natbib}
\usepackage{algorithm}
\usepackage{algpseudocode}
\usepackage{xcolor}

\newcommand{\df}{\mathrm{d}}
\newcommand{\X}{\mathsf{X}}
\newcommand{\Y}{\mathsf{Y}}
\newcommand{\B}{\mathcal{B}}
\newcommand{\tX}{\mathbf{X}}

\newcommand{\tx}{\mathbf{x}}
\newcommand{\ty}{\mathbf{y}}
\newcommand{\tz}{\mathbf{z}}
\newcommand{\tw}{\mathbf{w}}
\newcommand{\E}{\mathbb{E}}
\newcommand{\Prob}{\mathbb{P}}
\newcommand{\Range}{\mathcal{L}}
\newcommand{\RW}{R}
\newcommand{\RWNL}{\tilde{R}}
\newcommand{\ind}{\mathbf{1}}
\newcommand{\Id}{\mbox{Id}}

\newcommand{\pcite}[1]{\citeauthor{#1}'s \citeyearpar{#1}}

\newtheorem{theorem}{Theorem}
\newtheorem{lemma}[theorem]{Lemma}
\newtheorem{corollary}[theorem]{Corollary}
\newtheorem{proposition}[theorem]{Proposition}

\title{\bf Spectral Telescope: Convergence Rate Bounds for Random-Scan Gibbs Samplers Based on a Hierarchical Structure}

\author{Qian Qin \\ School of Statistics \\ University of Minnesota \and Guanyang Wang \\ Department of Statistics \\ Rutgers University}
\date{}

\begin{document}
	
	\maketitle
	
	\begin{abstract}
		Random-scan Gibbs samplers possess a natural hierarchical structure.
		The structure connects Gibbs samplers targeting higher dimensional distributions to those targeting lower dimensional ones.
		This leads to a quasi-telescoping property of their spectral gaps.
		Based on this property, we derive three new bounds on the spectral gaps and convergence rates of Gibbs samplers on general domains.
		The three bounds relate a chain's spectral gap to, respectively, the correlation structure of the target distribution, a class of random walk chains, and a collection of influence matrices.
		Notably, one of our results generalizes the technique of spectral independence, which has received considerable attention for its success on finite domains, to general state spaces.
		We illustrate our methods through a sampler targeting the uniform distribution on a corner of an $n$-cube.
	\end{abstract}

	\begin{center}
		\begin{minipage}{0.85\textwidth}
		\end{minipage}
	\end{center}

	\section{Introduction}
Gibbs samplers are among the most popular Markov chain Monte Carlo (MCMC) approaches to sample from  multivariate probability distributions. They have been applied and studied for sampling, counting, inference, and optimization in a variety of disciplines, including mathematics, statistics, physics, and computer science. This work concerns theoretical properties of  random-scan Gibbs samplers, also known as Glauber dynamics.	Our key observation is that these types of samplers possess a natural hierarchical, or recursive, structure that facilitates convergence analysis of the underlying Markov chains. 
Exploiting this structure, we derive a quasi-telescoping property for the spectral gaps of these chains, which leads several new convergence rate bounds.
	
	Our motivation mainly stems from the spectral independence technique recently developed in the theoretical computer science community, which we  now briefly review. Spectral independence was initially introduced in \cite{anari2021spectral} to establish a polynomial mixing time of the Gibbs sampler for hardcore models. 
	It has since received a tremendous
	amount of attention in computer science as it provides a powerful tool
	for proving fast, and sometimes optimal, mixing time bounds for Gibbs
	samplers for several important discrete models.
	It is particularly useful for samplers with many components, and despite being very recently developed, it is already regarded as an attractive alternative to more traditional tools for convergence analysis, such as Dobrushin's uniqueness condition.
	See \cite{feng2021rapid}, \cite{chen2021rapid}, \cite{jain2021spectral}, \cite{ chen2022spectral}, \cite{blanca2022mixing}, \cite{chen2022rapid}, and the references therein. 
	In the original framework, spectral
	independence was defined to bound the spectral gaps of samplers on the
	Boolean domain $\{0,1\}^n$, but it has since been improved and extended in various
	directions.  
	Some notable extensions include entropy factorization \citep{chen2021rapid, blanca2022mixing}, entropic independence \citep{anari2021entropic,anari2022entropic}, localization schemes \citep{chen2022localization}, and spectral independence on general finite domains \citep{feng2021rapid}.

	On continuous domains, convergence analysis of Gibbs samplers with many components remains challenging, despite impressive analyses for some interesting models \citep{roberts1997updating,smith2014gibbs, pillai2017kac, pillai2018mixing, janvresse2001spectral, carlen2003determination}.
	Practically speaking, the only existing
	tools that are designed specifically for convergence analysis of Gibbs
	samplers on general state spaces are based on the classical Dobrushin’s
	uniqueness condition \citep[see][and references therein]{wang2014convergence}.
	A framework that can be applied to chains outside finite domains thus seems ever so appealing.

	The main contribution of this paper is Theorem~\ref{thm:gap}, which describes the aforementioned quasi-telescoping property concerning the spectral gaps of Gibbs samplers on general state spaces.
	We refer to this property as ``the spectral telescope."
	Derived from a hierarchical structure of Gibbs samplers, the spectral telescope puts forward a flexible framework for bounding the spectral gap for these samplers on both discrete and continuous state spaces. 
	Based on it, we construct three types of bounds, given in Corollaries~\ref{cor:corr} to~\ref{cor:specind}.
	These three corollaries connect the spectral gap to, respectively, the correlation (dependence) structure of the target distribution, a collection of low-dimensional random walk chains, and a set of ``influence matrices" which define a spectral-independence-type condition.
	Corollaries~\ref{cor:randomwalk} and~\ref{cor:specind} extend/generalize results in \cite{alev2020improved} and \cite{feng2021rapid}, while Corollary~\ref{cor:corr} appears to be new even for finite state spaces.
	In particular, Corollary~\ref{cor:specind} generalizes \pcite{feng2021rapid} spectral-independence-based bound in two ways.
	Firstly, \pcite{feng2021rapid} result is extended from finite state spaces to general ones.
	Moreover, whereas \cite{feng2021rapid} calculate influence matrices based on total variation distances between conditional distributions, Corollary~\ref{cor:specind} uses influence matrices based on a more general class of Wasserstein divergences.
	Compared to methods based on the total variation distance, Wasserstein-based methods are often more effective for convergence analysis of Markov chains in high-dimensional settings \citep[see, e.g.,][]{hairer2011asymptotic,qin2020wasserstein}.

	
	Theorem~\ref{thm:gap} and its corollaries arm us with techniques
	to bound the spectral gap beyond those relying on spectral independence.
	These techniques can further be combined with various tools, such as orthogonal polynomials \citep[see, e.g.,][]{diaconis2008gibbs} and one-shot coupling \citep[see, e.g.,][]{roberts2002one}, to attain broader applicability.
	This is illustrated by a concrete example  in Section~\ref{sec:example}.
	In this example, we study a random-scan Gibbs sampler targeting the uniform distribution on the corner of an $n$-cube.
	We first invoke Corollary~\ref{cor:corr} and establish a tight spectral gap bound by analyzing the correlation structure of the target distribution using orthogonal polynomials. 
	In contrast, a straightforward generalization of spectral  independence where the influence matrices are calculated from total variation distances would give only trivial bounds. 
	A second non-trivial, but sub-optimal bound is obtained using Corollary~\ref{cor:specind}, where we utilize spectral independence based on suitable Wasserstein divergences.
	The example also shows that constructing a tight bound via our method requires adequate information on the target distribution.
	Our method is not a panacea, but rather one of the many steps towards understanding the convergence properties of Gibbs samplers. 
	Applying it to Gibbs samplers in various fields is a direction for future research.


	Properties similar to the spectral telescope have been derived for some models prior to our research.
	In particular, the spectral telescope is reminiscent of an inductive property of spectral gaps for Kac models, which are commonly used to study the distribution of physical particles \citep[][Theorem 2.2]{carlen2003determination}.
	The general mathematical setting in \cite{carlen2003determination} is quite different from ours, but some of the models they studied can be thought as Gibbs samplers whose target distributions satisfy certain symmetric properties.
	
	The rest of this paper is organized as follows. In the remainder of this section, we briefly  explain the hierarchical structure of Gibbs samplers, without getting into technical details. After introducing some preliminaries in Section~\ref{sec:prelim}, we formally define the Gibbs algorithm, describe its hierarchical structure, and state our main results in Section~\ref{sec:main}. Section~\ref{sec:example} contains the aforementioned example.
	The detailed proofs of our main results are provided in Section~\ref{sec:derivation}.

\subsection{The hierarchical structure: High level ideas} \label{ssec:highlevel}

Now we briefly explain the hierarchical (or recursive) structure of the random-scan Gibbs sampler, and defer the formal descriptions to Section \ref{sec:main}. 
Let $X_1,\dots,X_n$ be random elements whose joint distribution is~$\Pi$.
For $i \in [n] := \{1,\dots,n\}$ and~$x$ in the range of $X_i$, let $\Pi_{- \{i\} \mid \{i\}}(\cdot \mid x)$ denote the conditional distribution of 
\[
(X_1,\dots,X_{i-1},X_{i+1}, \dots,X_n) 
\]
given $X_i = x$.
Consider a Gibbs algorithm targeting~$\Pi$ that updates one component at a time.
Given the current state $(x_1,\dots,x_n)$, in each iteration, the sampler randomly and uniformly selects one component to update using its full conditional distribution.
Of course, selecting one component to update is equivalent to selecting $n-1$ components to fix.
This is, in turn, equivalent to selecting one component, say $x_i$, to fix, and then calling one step of the Gibbs sampler targeting $\Pi_{- \{i\} \mid \{i\}}(\cdot \mid x_i)$, which randomly selects $n-2$ of the remaining components to fix, and updates the component that was not selected.
We can then rewrite the Gibbs sampler as a recursive algorithm, as we can replace~$\Pi$ with $\Pi_{- \{i\} \mid \{i\}}(\cdot \mid x_i)$, and repeat the argument until the target distribution is univariate. 
	
	For illustration, suppose that $n = 4$, and the current state is $(x_1,\dots,x_4)$.
	One step of the Gibbs sampler targeting~$\Pi$ proceeds as follows:
	\begin{enumerate}
		\item Randomly and uniformly select an index~$j$ from $[4] = \{1,2,3,4\}$.
		\item Update $x_j$.
	\end{enumerate}
	This is equivalent to the following procedure:
	\begin{enumerate}
		\item[1'.] Randomly and uniformly select an index $i_1$ from $[4]$.
		\item[2'.] Randomly and uniformly select an index $i_2$ from $[4] \setminus \{i_1\}$.
		\item[3'.] Randomly and uniformly select an index $i_3$ from $[4] \setminus \{i_1,i_2\}$.
		\item[4'.] Update $x_j$, where $\{j\} = [4] \setminus \{i_1,i_2,i_3\}$.
	\end{enumerate}
	Note the hierarchical structure:
	Steps 2'-4' form one step of the Gibbs sampler targeting $\Pi_{-\{i_1\} \mid \{i_1\}}(\cdot \mid x_{i_1})$.
	Step 3'-4' make up one step of the Gibbs sampler targeting the conditional distribution of $X_{i_3}$ and $X_j$ given $X_{i_1} = x_{i_1}$ and $X_{i_2} = x_{i_2}$.
	Finally, step 4' alone can be regarded as one step of the Gibbs sampler targeting the conditional distribution of $X_j$ given all other components.
	

As we will see,	this hierarchical structure  not only  reformulates the original Gibbs sampler, but also  leads to non-trivial bounds on the spectral gap. 
	

	\section{Preliminaries}\label{sec:prelim}
	
	Consider a probability space $(E, \mathcal{F}, \nu)$.
	We use $L^2(\nu)$ to denote the set of measurable functions $f: E \to \mathbb{R}$ such that 
	\[
	\int_E f(x)^2 \, \nu(\df x) < \infty.
	\]
	For $f, g \in L^2(\nu)$, one can define their inner product
	\[
	\langle f, g \rangle_{\nu} = \int_E f(x) g(x) \, \nu(\df x).
	\]
	In particular, the $L^2$ norm of a function $f \in L^2(\nu)$ is $\|f\|_{\nu} = \sqrt{\langle f, f \rangle_{\nu}}$.
	Two functions in $L^2(\nu)$ are equal if their difference has a vanishing $L^2$ norm.
	$L^2(\nu)$ forms a Hilbert space.
	We use $L_0^2(\nu)$ to denote the subspace of $L^2(\nu)$ consisting of functions~$f$ such that
	\[
	\langle f, 1 \rangle_{\nu} = \int_E f(x) \, \nu(\df x) = 0.
	\]
	Also, $L_*^2(\nu)$ is used to denote the set of probability measures $\omega: \mathcal{F} \to [0,1]$ such that~$\omega$ is absolutely continuous with respect to~$\nu$, and that $\df \omega/ \df \nu \in L^2(\nu)$.
	For $\omega_1, \omega_2 \in L_*^2(\nu)$, their $L^2$ distance is
	\[
	\|\omega_1 - \omega_2 \|_{\nu} =  \left\| \frac{\df \omega_1}{\df \nu} - \frac{\df \omega_2}{\df \nu} \right\|_{\nu}.
	\]

	Let $K: E \times \mathcal{F} \to [0,1]$ be a transition kernel that describes the transition law of a Markov chain $(X(t))_{t=0}^{\infty}$.
	We say that~$\nu$ is a stationary distribution of $(X(t))$ if
	\[
	\nu K (\cdot) := \int_{E} K(x,\cdot) \, \nu(\df x) = \nu(\cdot).
	\]
	Suppose that $\nu K = \nu$.
	For $f \in L^2(\nu)$ and $x \in E$, define
	\[
	Kf(x) = \int_{E} f(x') K(x, \df x').
	\]
	If $f \in L_0^2(\nu)$, then $Kf \in L_0^2(\nu)$.
	Then we can view~$K$ as a linear operator on $L_0^2(\nu)$, referred to as a Markov operator.
	Using Cauchy-Schwarz, one can show that
	\[
	\|K\|_{\nu} := \sup_{\stackrel{f \in L_0^2(\nu)}{f \neq 0}} \frac{\|Kf\|_{\nu}}{\|f\|_{\nu}} \leq 1,
	\]
	where $\|K\|_{\nu}$ is called the $L^2$ norm of~$K$.

	The above framework is particularly useful in the study of reversible chains.
	A chain associated with~$K$, where $\nu K = \nu$, is said to be reversible with respect to~$\nu$ if~$K$ is self-adjoint, i.e., for $f, g \in L_0^2(\nu)$,
	\[
	\langle Kf, g \rangle_{\nu} = \langle f, Kg \rangle_{\nu}.
	\]
	All chains studied in this paper are reversible with respect to their respective stationary distributions.
	Suppose that~$K$ is self-adjoint.
	Then the spectral gap of~$K$ (or that of a chain associated with~$K$) is $1 - \|K\|_{\nu}$.
	The magnitude of the spectral gap governs how fast a Markov chain associated with~$K$ converges to its stationary distribution~$\nu$, with a larger gap indicating faster convergence.
	Indeed, the following well-known result states that $\|K\|_{\nu}$ is in fact the $L^2$ convergence rate of the chain.
	\begin{lemma} \citep{roberts1997geometric} \label{lem:roberts}
		Let $(X(t))_{t=0}^{\infty}$ be a chain reversible with respect to~$\nu$ and let~$K$ be its Markov operator.
		For $\omega \in L_*^2(\nu)$ and $t \geq 0$, let $\omega K^t$ be the distribution of $X(t)$ if $X(0) \sim \omega$.
		Then, for $\rho < 1$, $\|K\|_{\nu} \leq \rho$ if and only if the following holds:
		For $\omega \in L_*^2(\nu)$, there exists a constant $C_{\omega} < \infty$ such that, for $t \geq 0$,
		\[
		\|\omega K^t - \nu\|_{\nu} \leq C_{\omega} \rho^t.
		\]
	\end{lemma}
	
	The Markov operator~$K$ is said to be positive semi-definite if it is self-adjoint, and $\langle f, Kf \rangle_{\nu} \geq 0$ for $f \in L_0^2(\nu)$.
	In this case, the following formula holds:
	\[
	\|K\|_{\nu} = \sup_{\stackrel{f \in L_0^2(\nu)}{f \neq 0}} \frac{\langle f, Kf \rangle_{\nu} }{\|f\|_{\nu}^2}.
	\]
	It is well-known \citep[see, e.g.,][]{liu1995covariance} that operators of random-scan Gibbs algorithms are positive semi-definite.

	\section{A Hierarchical Structure} \label{sec:main}
	
	\subsection{Gibbs samplers and their recursive forms}
	
	Let $(\X_1,\B_1,\mu_1), \dots, (\X_n,\B_n,\mu_n)$ be $\sigma$-finite measure spaces, where~$n$ is a positive integer that is at least~2.
	Assume that in each space, singletons are measurable.
	Suppose that, for $i=1,\dots,n$, $X_i$ is an $\X_i$-valued random element, and that $\tX = (X_1,\dots,X_n)$ has a joint distribution~$\Pi$.
	
	Assume that~$\Pi$  is absolutely  continuous with respect to the base measure $\mu_1 \times \dots \times \mu_n$  with Radon-Nikodym derivative (density) $\pi$, so that~$\pi$ is a measurable function on $\X := \X_1 \times \dots \times \X_n$.
	Although Radon-Nikodym derivatives only need to be defined outside a null set, for concreteness we insist that~$\pi$ is specified everywhere on~$\X$.
	While these assumptions would seem more rigid than necessary, they bring a great deal of technical and notational convenience.
	For a nonempty set of indices $\Gamma= \{\gamma_1,\dots,\gamma_{|\Gamma|}\} \subset [n]$, where $\gamma_1 < \dots < \gamma_{|\Gamma|}$, let $\X_{\Gamma} = \X_{\gamma_1} \times \dots \times \X_{\gamma_{|\Gamma|}}$, 
	$\mu_{\Gamma} = \mu_{\gamma_1} \times \dots \times \mu_{\gamma_{|\Gamma|}}$, and $\tX_{\Gamma} = (X_{\gamma_1}, \dots, X_{\gamma_{|\Gamma|}})$.
	Also, for $\Gamma$ given above and $\tx = (x_1,\dots,x_n) \in \X$, where $x_i \in \X_i$ for each~$i$, let $\tx_{\Gamma} = (x_{\gamma_1}, \dots, x_{\gamma_{|\Gamma|}})$.
	For any $\Gamma$ such that $1 \leq |\Gamma| \leq n-1$, the marginal density of $\tX_{\Gamma}$ evaluated at any $\ty \in \X_{\Gamma}$ is 
	\[
	\pi_{\Gamma}(\ty) = \int_{\X_{- \Gamma}} \pi(\tx) \, \df \tx_{-\Gamma}, \quad \text{where } \tx \in \X \text{ satisfies } \tx_{\Gamma} = \ty.
	\]
	In the above equation, $-\Gamma = [n] \setminus \Gamma$, and $\df \tx_{-\Gamma}$ is a short-hand notation for $\mu_{- \Gamma}( \df \tx_{- \Gamma} )$.
	By convention, $\pi_{[n]} = \pi$.
	For nonempty sets $\Lambda, \Gamma \subset [n]$ such that $\Lambda \cap \Gamma = \emptyset$, the conditional density of $\tX_{\Gamma}$ given $\tX_{\Lambda} = \ty \in \X_{\Lambda}$, denoted by $\pi_{\Gamma \mid \Lambda}(\cdot \mid \ty)$, is defined for $\ty \in \X_{\Lambda}$ such that $\pi_{\Lambda}(\ty) > 0$, and given by
	\[
	\pi_{\Gamma \mid \Lambda}(\tz \mid \ty) = \frac{\pi_{\Lambda \cup \Gamma}( \tx_{\Lambda \cup \Gamma} )}{\pi_{\Lambda}(\ty)}, \quad \text{where } \tx \in \X \text{ satisfies } \tx_{\Lambda} = \ty \text{ and } \tx_{\Gamma} = \tz.
	\]
	If $\pi_{\Lambda}(\ty) = 0$, we let $\pi_{\Gamma \mid \Lambda}(\cdot \mid \ty)$ be an arbitrary probability density function on $\X_{\Gamma}$.
	By convention, if $\Lambda = \emptyset$, $\pi_{\Gamma \mid \Lambda}(\cdot \mid \ty)$ means $\pi_{\Gamma}(\cdot)$, even though $\ty \in \X_{\emptyset}$ cannot be specified.

	A random-scan Gibbs sampler targeting~$\pi$ with block size~$l$ is described in Algorithm~\ref{alg:gibbs}.
	In short, given the current state $\tx \in \X$, the sampler randomly selects a subset~$\Gamma$ of indices, and updates $\tx_{\Gamma}$ using the conditional distribution of $\tX_{\Gamma}$ given $\tX_{- \Gamma} = \tx_{- \Gamma}$.
	In many applications, a block size of~1 is used because it becomes more difficult to draw from the corresponding full conditional distributions when the block size increases.
	Regardless of the block size, the underlying Markov chain is reversible with respect to~$\pi$.
	
	\setcounter{algorithm}{0}
	
	\begin{algorithm}
		\caption{One step of the Gibbs sampler targeting~$\pi$, block size~$l$, where $l \in \{1,\dots,n\}$:}\label{alg:gibbs}
		\begin{algorithmic}
			\State \textbf{Input:} Current state $\tx \in \X$.
			\State Randomly and uniformly choose a subset of indices $\Gamma \subset [n]$ under the constraint $|\Gamma| = l$.
			\State Draw $\tw \in \X_{\Gamma}$ from $\pi_{\Gamma \mid - \Gamma}(\cdot \mid \tx_{- \Gamma})$.
			\State Let $\tx' \in \X$ be such that $\tx'_{\Gamma} = \tw$ and $\tx'_{- \Gamma} = \tx_{- \Gamma}$.
			\State \textbf{Return:} Next state $\tx'$.
		\end{algorithmic}
	\end{algorithm}
	
	Algorithm~\ref{alg:gibbs} is a special case of Algorithm~\ref{alg:blo-gibbs}, which follows the same procedure, but targets $\pi_{- \Lambda \mid \Lambda}(\cdot \mid \ty)$ for some $\Lambda \subset [n]$ such that $|\Lambda| \leq n-1$ and $\ty \in \X_{\Lambda}$.
	When $\pi_{\Lambda}(\ty) > 0$, the underlying Markov chain is reversible with respect to $\pi_{- \Lambda \mid \Lambda}(\cdot \mid \ty)$.
	Taking $\Lambda = \emptyset$ in Algorithm~\ref{alg:blo-gibbs} yields Algorithm~\ref{alg:gibbs}.

	\begin{algorithm}
		\caption{One step of the Gibbs sampler targeting $\pi_{- \Lambda \mid \Lambda}(\cdot \mid \ty)$, block size $l$, where $l \in \{1,\dots,n-|\Lambda|\}$:}\label{alg:blo-gibbs}
		\begin{algorithmic}
			\State \textbf{Input:} Current state $\tz \in \X_{- \Lambda}$.
			\State Let $\tx \in \X$ be such that $\tx_{\Lambda} = \ty$ and $\tx_{- \Lambda} = \tz$.
			\State Randomly and uniformly choose a set of indices $\Gamma \subset - \Lambda$ under the constraint $|\Gamma| = l$.
			\State Draw $\tw \in \X_{\Gamma}$ from $\pi_{\Gamma \mid - \Gamma}(\cdot \mid \tx_{- \Gamma})$.
			\State Let $\tx' \in \X$ be such that $\tx'_{\Gamma} = \tw$ and $\tx'_{- \Gamma} = \tx_{- \Gamma}$.
			\State \textbf{Return:} New state $\tz' = \tx'_{- \Lambda}$.
		\end{algorithmic}
	\end{algorithm}
	
	\begin{algorithm}[htbp]
		\caption{One step of the recursive Gibbs sampler targeting $\pi_{- \Lambda \mid \Lambda}(\cdot \mid \ty)$, block size~$l$, where $l \in \{1,\dots, n - |\Lambda|\}$:}\label{alg:rec-gibbs}
		\begin{algorithmic}
			\State \textbf{Input:} Current state $\tz \in \X_{- \Lambda}$.
			\State Let $\tx \in \X$ be such that $\tx_{\Lambda} = \ty$ and $\tx_{- \Lambda} = \tz$.
			
			\If{$|\Lambda| = n-l$}
			\State Draw $\tz' \in \X_{- \Lambda}$ from $\pi_{- \Lambda \mid \Lambda}(\cdot \mid \tx_{\Lambda})$.
			\State \textbf{Return:} New state $\tz'$.

			\Else
			
			\State Randomly and uniformly choose a coordinate $i \in - \Lambda$.
			\State Draw $\tw \in \X_{- (\Lambda \cup \{i\})}$ by running one step of the recursive Gibbs sampler targeting $\pi_{- (\Lambda \cup \{i\}) \mid \Lambda \cup \{i\}}(\cdot \mid \tx_{\Lambda \cup \{i\}})$ with block size~$l$ and current state $\tx_{- (\Lambda \cup \{i\})}$.
			\State Let $\tx' \in \X$ be such that $\tx'_{- (\Lambda \cup \{i\})} = \tw$ and $\tx'_{\Lambda \cup \{i\}} = \tx_{\Lambda \cup \{i\}}$.
			\State \textbf{Return:} New state $\tz' = \tx'_{- \Lambda}$.
			
			\EndIf
		\end{algorithmic}
	\end{algorithm}
	
	Our analysis begins with the observation that Algorithm~\ref{alg:blo-gibbs} has a hierarchical, or recursive structure.
	Indeed, following arguments given in Section~\ref{ssec:highlevel}, we see that Algorithm~\ref{alg:blo-gibbs} can be written into a recursive form as in Algorithm~\ref{alg:rec-gibbs}.
	In particular, Algorithm~\ref{alg:gibbs} is equivalent to Algorithm~\ref{alg:rec-gibbs} for $\Lambda = \emptyset$.

	Consider the significance of the recursive representation.
	It connects the Gibbs sampler targeting $\pi_{- \Lambda \mid \Lambda}$ to ones targeting lower dimensional distributions, given by $\pi_{- \Lambda' \mid \Lambda'}$ where $\Lambda' \supset \Lambda$.
	While one would rarely implement the recursive algorithm in practice, based on it we can establish multiple intriguing relations concerning the convergence rate and spectral gap of the standard Algorithm~\ref{alg:gibbs}.
	We now list these relations.
	The detailed derivation is given in Section~\ref{sec:derivation}.

	\subsection{The spectral telescope}
	For $\Lambda \subset [n]$ such that $|\Lambda| \in \{0,\dots,n-1\}$, $\ty \in \X_{\Lambda}$, and $l \in \{1,\dots,n-|\Lambda|\}$, 
	let $\mbox{gap}  (\Lambda,\ty,l)$ be the spectral gap associated with Algorithm~\ref{alg:blo-gibbs}.
	For $l \in \{1,\dots,n\}$ and $m \in \{l,\dots,n\}$, let
	\[
	\mbox{Gap} (m,l) = \min_{\stackrel{\Lambda \subset [n]}{|\Lambda| = n-m}}  \inf_{\stackrel{\ty \in \X_{\Lambda}}{\pi_{\Lambda}(\ty) > 0}} \mbox{gap} (\Lambda, \ty, l).
	\]
	In particular, $\mbox{Gap}(n,l)$ is simply the spectral gap of Algorithm~\ref{alg:gibbs}.
	Our main result is a consequence of the hierarchical structure of Gibbs samplers.
	\begin{theorem} [Spectral Telescope] \label{thm:gap}
		For $l \in \{1,\dots,n-1\}$ and $m \in \{l+1,\dots,n\}$,
		\[
		\mbox{Gap}  (m,l) \geq \mbox{Gap}  (m,m-1)  \mbox{Gap} (m-1,l).
		\]
		In particular, for $l \in \{1,\dots,n-1\}$,
		\[
		\mbox{Gap} (n,l) \geq \prod_{m=l+1}^{n} \mbox{Gap} (m,m-1).
		\]
	\end{theorem}

	Theorem~\ref{thm:gap} describes a quasi-telescoping property of the sequence $\mbox{Gap}(n,n-1), \dots, \mbox{Gap}(2,1)$.
	We dub it ``the spectral telescope."
	From here we see that it is possible to bound $\mbox{Gap}(n,l)$ from below via lower bounds on $\mbox{Gap}(m,m-1)$ for $m \in \{l+1,\dots,n\}$.
	All other major results in this section are obtained via this strategy.

	\subsection{Spectral gaps and correlation coefficients} \label{sec:correlation}
	
	Let $(Y_1,\dots,Y_m)$ be a vector of random elements taking values in a product space $\Y_1 \times \dots \times \Y_m$.
	For $i = 1,\dots,m$, let $\varpi_i$ be the marginal distribution of $Y_i$, and note that $L_0^2(\varpi_i)$ represents the collection of real functions~$f$ on $\Y_i$ such that 
	\[
	\E[f(Y_i)] = \int_{\Y_i} f(y) \, \varpi_i(\df y) = 0, \quad \E[f(Y_i)^2] = \int_{\Y_i} f(y)^2 \, \varpi_i(\df y) < \infty.
	\]
	Define the summation-based correlation coefficient of $(Y_1,\dots,Y_m)$ to be
	\[
	s_*(Y_1,\dots,Y_m) = \sup_{\stackrel{f_i \in L_0^2(\varpi_i) \; \forall i}{\exists i \text{ s.t. } \E[f_i(Y_i)^2] > 0}} \frac{\E \left\{ \left[ \sum_{i=1}^{m} f_i(Y_i) \right]^2 \right\}}{m \sum_{i=1}^m \E [ f_i(Y_i)^2 ] }.
	\]
	$s_*(Y_1,\dots,Y_m)$ can range from $1/m$ to~1.
	To establish the lower bound, take $f_i = 0$ for $i \geq 2$.
	To establish the upper bound, use Cauchy-Schwarz. 
	If $Y_1,\dots,Y_m$ are independent, then this coefficient is $1/m$.
	If there exists a sequence of functions $f_1,\dots,f_m$ such that $f_i \in L_0^2(\varpi_i)$ and $f_i \neq 0$ for $i \in [m]$, and $f_i(Y_i) = f_j(Y_j)$ for $i,j \in [m]$, then the correlation coefficient is~1. 
	
	For $\Lambda \subset [n]$ such that $|\Lambda| \leq n-1$ and $\ty \in \X_{\Lambda}$, let
	\[
	s(\Lambda,\ty) = s_*(Y_1,\dots,Y_{n-|\Lambda|}),
	\]
	where $(Y_1,\dots,Y_{n-|\Lambda|})$ is distributed according to $\pi_{- \Lambda \mid \Lambda}(\cdot \mid \ty)$, i.e., the conditional distribution of $\tX_{- \Lambda}$ given $\tX_{\Lambda} = \ty$.
	For $m \in \{2,\dots,n\}$, let
	\[
	S(m) = \max_{\stackrel{\Lambda \subset [n]}{|\Lambda| = n-m}} \sup_{\stackrel{\ty \in \X_{\Lambda}}{\pi_{\Lambda}(\ty) > 0}} s(\Lambda,\ty).
	\]
	Then the following holds.
	\begin{corollary} \label{cor:corr}
		For $m \in \{2,\dots,n\}$, 
		\[
		\mbox{Gap}(m,m-1) \geq 1 - S(m).
		\]
		As a result, by Theorem~\ref{thm:gap}, for $l \in \{1,\dots,n-1\}$,
		\[
		\mbox{Gap} (n,l) \geq \prod_{m=l+1}^n [1 - S(m)].
		\]
	\end{corollary}
	This result relates the convergence rate of Algorithm~\ref{alg:gibbs} to the dependence structure of~$\Pi$.

	\subsection{Spectral gaps and random walks}
	Let $\Lambda \subset [n]$ be such that $|\Lambda| \leq n-1$, and let $\ty \in \X_{\Lambda}$.
	We can define a random walk on the space $\bigcup_{i \in - \Lambda} (\{i\} \times \X_i)$, given by Algorithm~\ref{alg:randomwalk}.
	Whenever $\pi_{\Lambda}(\ty) > 0$, this is a Markov chain reversible with respect to the probability measure $\varphi_{\Lambda,\ty}$ given by
	\[
	\varphi_{\Lambda,\ty}(\{i\} \times A) = \frac{1}{n-|\Lambda|} \int_A \pi_{ \{i\} \mid \Lambda }(x \mid \ty) \, \df x, \quad i \in - \Lambda, \; A \in \B_i,
	\]
	where $\df x$ is a short-hand notation for $\mu_i(\df x)$.
	Let $g(\Lambda,\ty)$ be the spectral gap of this chain.
	For $m \in \{2,\dots,n\}$, let
	\[
	G(m) = \min_{\stackrel{\Lambda \subset [n]}{|\Lambda| = n-m}} \inf_{\stackrel{\ty \in \X_{\Lambda}}{\pi_{\Lambda}(\ty) > 0}} g(\Lambda,\ty).
	\]
	We then have the following result:
	
	\begin{algorithm}
		\caption{One step of a random walk associated with $\Lambda \subset [n]$ and $\ty \in \X_{\Lambda}$:}\label{alg:randomwalk}
		\begin{algorithmic}
			\State \textbf{Input:} Current state $(j,x) \in \bigcup_{i \in - \Lambda} (\{i\} \times \X_i)$.
			\State Let $\tx \in \X$ be such that $\tx_{\Lambda} = \ty$ and $\tx_{\{j\}} = x$.
			\State Randomly and uniformly choose a coordinate $j' \in - \Lambda$.
			
			\If{$j'=j$}
			\State Set $x' = x$.

			\Else
			
			\State Draw $x' \in \X_{\{j'\}}$ from $\pi_{\{j'\} \mid \Lambda \cup \{j\}} (\cdot \mid \tx_{\Lambda \cup \{j\}})$.
			
			\EndIf
			
			\State \textbf{Return:} New State $(j',x')$.
			
		\end{algorithmic}
	\end{algorithm}
	
	\begin{corollary} \label{cor:randomwalk}
		For $m \in \{2,\dots,n\}$, 
		\[
		\mbox{Gap}(m,m-1) \geq G(m).
		\]
		As a result, by Theorem~\ref{thm:gap}, for $l \in \{1,\dots,n-1\}$,
		\[
		\mbox{Gap}  (n,l) \geq \prod_{m=l+1}^n G(m).
		\]
	\end{corollary}
	
	This extends a result in~\cite{alev2020improved}, which concerns random walks on pure simplical complexes, a discrete structure frequently studied in computer science.

	\subsection{Spectral independence} \label{ssec:specind}
	
	Let $\Lambda \subset [n]$ be such that $|\Lambda| \leq n-2$, and let $\ty \in \X_{\Lambda}$ be such that $\pi_{\Lambda}(\ty) > 0$.
	Suppose that, for $i \in - \Lambda$, there is a measurable ``distance-like" function $d_{\Lambda,\ty,i}: \X_i \times \X_i \to [0,\infty)$ such that (i) $x=y$ if and only if $d_{\Lambda,\ty,i}(x,y) = 0$, and (ii) $d_{\Lambda,\ty,i}(x,y) = d_{\Lambda,\ty,i}(y,x)$ for $x,y \in \X_i$.
	For $j \in - \Lambda$ such that $i \neq j$ and $x \in \X_i$, let $\Pi^j_{\Lambda,\ty,i,x}$ be the probability measure associated with $\pi_{\{j\} \mid \Lambda \cup \{i\}}(\cdot \mid \tx_{\Lambda \cup \{i\}})$, where $\tx_{\Lambda} = \ty$ and $\tx_{\{i\}} = x$.
	In other words, $\Pi^j_{\Lambda,\ty,i,x}$ is the conditional distribution of $X_j$ given $X_i = x$ and $\tX_{\Lambda} = \ty$.
	Assume that the following conditions hold:
	\begin{enumerate}
		\item [(H1)] 
		For $i \in - \Lambda$,
		\[
		\int_{\X_i} \left[ \int_{\X_i} d_{\Lambda,\ty,i}(x,x') \, \pi_{\{i\} \mid \Lambda} (\df x \mid \ty) \right]^2 \pi_{\{i\} \mid \Lambda} (\df x' \mid \ty) < \infty.
		\]
		\item [(H2)] There exists $k < \infty$ such that, for $i, j \in - \Lambda$ satisfying $i \neq j$ and $\pi_{\{i\} \mid \Lambda}(\cdot \mid \ty)$-almost every $x, x' \in \X_i$, 
		\[
		d_{\scriptsize\mbox{TV}}(\Pi^j_{\Lambda,\ty,i,x}, \Pi^j_{\Lambda,\ty,i,x'}) \leq k d_{\Lambda,\ty,i}(x,x'),
		\]
		where $d_{\scriptsize\mbox{TV}}$ denotes the total variation distance, which is the maximal difference between the probabilities of a measurable set assigned by the two probability measures.
		The constant $k$ may depend on $(\Lambda,\ty)$ but not on $(i,j,x,x')$.
	\end{enumerate}
	Note that if, for $i \in - \Lambda$, $d_{\Lambda,\ty,i}$ is the discrete metric, i.e., $d_{\Lambda,\ty,i}(x,x') = \ind_{x \neq x'}$, then (H1) and (H2) are satisfied.
	
	For two probability distributions $\nu_1$ and $\nu_2$ on $\B_i$, where $i \in [n]$, denote by $C(\nu_1,\nu_2)$ the collection of couplings of $\nu_1$ and $\nu_2$.
	That is, $\nu \in C(\nu_1,\nu_2)$ if and only if~$\nu$ is a probability measure on $\B_i \times \B_i$ such that $\nu(A \times \X_i) = \nu_1(A)$ and $\nu(\X_i \times A) = \nu_2(A)$ for $A \in \B_i$.

	A coupling kernel associated with $(\Lambda,\ty,i,j)$, where $i,j \in - \Lambda$ and $i \neq j$, is a Markov transition kernel $K_{i,j}: \X_i \times \X_i \to \B_j \times \B_j$ such that $K_{i,j}((x,x'), \cdot)$ is a probability measure in $C(\Pi^j_{\Lambda,\ty,i,x}, \Pi^j_{\Lambda,\ty,i,x'})$ for $x,x' \in \X_i$.
	(Of course, $K_{i,j}$ also depends on~$\Lambda$ and~$\ty$, but to suppress notation we do not include them in the subscript.
	The same goes for $\phi_{i,j}$ given below.)
	We say that a contraction condition holds for $(\Lambda, \ty, i, j)$ with coefficient $\phi_{i,j} \in [0,\infty)$ if there exists a coupling kernel $K_{i,j}$ associated with $(\Lambda,\ty,i,j)$ such that
	\begin{equation} \label{ine:contraction}
	\int_{\X_j \times \X_j} d_{\Lambda,\ty,j}(x'', x''') \, K_{i,j} \left((x,x'), \df (x'',x''') \right) \leq \phi_{i,j} \, d_{\Lambda,\ty,i}(x,x')
	\end{equation}
	for $\pi_{\{i\} \mid \Lambda}(\cdot \mid \ty)$-almost every $x,x' \in \X_i$.
	Note that~\eqref{ine:contraction} implies that the Wasserstein divergence induced by $d_{\Lambda,\ty,j}$ between $\Pi^j_{\Lambda,\ty,i,x}$ and $\Pi^j_{\Lambda,\ty,i,x'}$ is upper bounded by $\phi_{i,j} \, d_{\Lambda,\ty,i}(x,x')$.
	In particular, if $d_{\Lambda,\ty,i}$ and $d_{\Lambda,\ty,j}$ are the discrete metric, then~\eqref{ine:contraction} is equivalent to contraction in the total variation distance, i.e.,
	\begin{equation} \nonumber
		d_{\scriptsize\mbox{TV}}(\Pi^j_{\Lambda,\ty,i,x}, \Pi^j_{\Lambda,\ty,i,x'}) \leq \phi_{i,j} \ind_{\tx \neq \tx'}.
	\end{equation}
	
	An influence matrix associated with $(\Lambda,\ty)$, denoted by $\Phi(\Lambda,\ty)$, is a square matrix of dimension $n - |\Lambda|$ whose $i,j$th element (where $i \neq j$) is the contraction coefficient $\phi_{i,j}$ given above, assuming that a contraction condition holds for $(\Lambda,\ty,i,j)$.
	The diagonal elements of $\Phi(\Lambda,\ty)$ are set to be zero.
	Let $r(\Phi(\Lambda,\ty))$ be the spectral radius of $\Phi(\Lambda,\ty)$.
	
	Now, allow~$\Lambda$ and~$\ty$ to vary.
	Given $l \in \{1,\dots,n-1\}$ and $(\eta_{l+1}, \dots, \eta_n) \in \mathbb{R}^{n-l}$ such that $\eta_m < m-1$ for each~$m$, we say that the full joint distribution~$\Pi$ is $(\eta_{l+1}, \dots, \eta_n)$-spectrally independent if the following holds: 
	For every $m \in \{l+1,\dots,n\}$, $\Lambda \subset [n]$ such that $|\Lambda| = n-m$, and $\ty \in \X_{\Lambda}$ such that $\pi_{\Lambda}(\ty) > 0$, there exists an influence matrix $\Phi(\Lambda,\ty)$ associated with $(\Lambda,\ty)$ such that 
	$
	r(\Phi(\Lambda,\ty)) \leq \eta_m.
	$
	
	Recently, spectral independence has received tremendous attention in the theoretical computer science community.
	It is regarded as a potentially powerful tool for bounding the spectral gaps of Gibbs chains.
	All existing works on this topic focus on chains on finite state spaces.
	Moreover, the distance-like function $d_{\Lambda,\ty,i}$ is always set to be the discrete metric.
	Our next corollary extends existing results, particularly \pcite{feng2021rapid} Theorem 3.1, with regard to these two aspects.

	\begin{corollary} \label{cor:specind}
		Let $m \in \{2,\dots,n\}$.
		Suppose that, for $\Lambda \subset [n]$ such that $n - |\Lambda| = m$ and $\ty \in \X_{\Lambda}$ such that $\pi_{\Lambda}(\ty) > 0$, there is an influence matrix $\Phi(\Lambda,\ty)$ associated with $(\Lambda,\ty)$ such that 
		\[
		r(\Phi(\Lambda,\ty)) \leq \eta,
		\]
		where $\eta < m-1$.
		Then
		\[
		\mbox{Gap}(m,m-1) \geq \frac{m-1}{m} - \frac{\eta}{m}.
		\]
		In particular, it follows from Theorem~\ref{thm:gap} that, for $l \in \{1,\dots,n-1\}$, if~$\Pi$ is $(\eta_{l+1},\dots,\eta_n)$-spectrally independent, then
		\[
		\mbox{Gap}(n,l) \geq \prod_{k=l+1}^n \left( \frac{k-1}{k} - \frac{\eta_k}{k} \right) = \frac{l}{n} \prod_{k=l+1}^n \left( 1 - \frac{\eta_k}{k-1} \right) .
		\]
	\end{corollary}
	
%

	As will be seen from Section~\ref{sec:derivation}, Corollary~\ref{cor:specind} is derived from Corollary~\ref{cor:randomwalk}, which is in turn derived from Corollary~\ref{cor:corr}, which is in turn derived from Theorem~\ref{thm:gap}.
	
	\subsection{Additional remarks}
	
 We observe the lower bounds on $\mbox{Gap}(n,l)$ in Corollaries~\ref{cor:corr} and~\ref{cor:specind} are at most $l/n$.
The quantity	$l/n$ is in fact an upper bound on the spectral gap of the random-scan Gibbs sampler targeting~$\pi$ with block size~$l$.
	To see this, let~$K$ be the Markov operator associated with the algorithm.
	Then, for $f \in L_0^2(\Pi)$ and $\tx \in \X$,
	\[
	Kf( \tx ) = \frac{1}{{n \choose l}} \sum_{\stackrel{\Gamma \subset [n]}{|\Gamma| = l}} \E \left[ f(\tX) \mid \tX_{- \Gamma} = \tx_{- \Gamma} \right].
	\]
	Let $f \in L_0^2(\Pi)$ be such that $\|f\|_{\Pi} = \E[f(\tX)^2] = 1$.
	Suppose that $f(\tx)$ depends on $\tx \in \X$ only through $\tx_{\{1\}}$.
	One can verify that, whenever $n \geq l+1$,
	\[
	\begin{aligned}
		\langle f, K f \rangle_{\Pi} &= \frac{1}{{n \choose l}} \sum_{\stackrel{\Gamma \subset [n]}{|\Gamma| = l}} \E \left\{ \E \left[ f(\tX) \mid \tX_{- \Gamma} \right]^2 \right\} \\
		&\geq  \frac{1}{{n \choose l}} \sum_{\stackrel{\Gamma \subset [n]}{|\Gamma| = l, \; 1 \not\in \Gamma}} \E \left\{ \E \left[ f(\tX) \mid \tX_{- \Gamma}  \right]^2 \right\} \\
		&= \frac{1}{{n \choose l}} \sum_{\stackrel{\Gamma \subset [n]}{|\Gamma| = l, \; 1 \not\in \Gamma}} \E[f(\tX)^2] \\
		&= \frac{n-l}{n}.
	\end{aligned}
	\]
	Then the spectral gap satisfies
	\[
	1 - \|K\|_{\Pi} \leq 1 - \langle f, K f \rangle_{\Pi} \leq \frac{l}{n}.
	\]
	
	Our framework leaves several interesting open questions and directions for further extension. It is unclear when the lower bound in Theorem \ref{thm:gap} will give the exact spectral gap. 
	Moreover, the existing bound relies on uniform lower bounds on the spectral gap of lower-dimensional Gibbs samplers.  Generalizing  existing results to position-dependent lower bounds may increase the applicability of our method. 
	
	Perhaps more importantly, in many models,~$\Pi$ does not have a Radon-Nikodym derivative~$\pi$.
	It seems that many of our results could still hold if we replace the existence of~$\pi$ with some weaker regularity conditions.
	However, establishing this rigorously would likely require extremely careful (and possibly tedious) analysis.
	This is an important topic for future studies.

	\section{An Example}\label{sec:example}
	
	The relations derived in Section~\ref{sec:main} can be used to construct convergence bounds for Gibbs algorithms.
	The following example illustrates the strengths and limitations of this framework.

	Let $\X_1 = \dots = \X_n = (0,1)$, and let $\mu_1 = \dots = \mu_n$ be Lebesgue measures.
	Let
	\[
	\pi(x_1,\dots,x_n) \propto \begin{cases}
		1 & \sum_{i=1}^n x_i < 1, \\
		0 & \text{otherwise}.
	\end{cases}
	\]
	That is,~$\pi$ corresponds to the uniform distribution on the corner of an $n$-cube given by
	\[
	\left\{ (x_1,\dots,x_n) \in (0,1)^n: \; \sum_{i=1}^n x_i < 1 \right\}.
	\]
	Consider the random-scan Gibbs algorithm targeting~$\pi$ with block size $l = 1$.
	In each iteration of the algorithm, given the current state $\tx = (x_1,\dots,x_n) \in \X= (0,1)^n$, where $\sum_{i=1}^n x_i < 1$, one randomly and uniformly selects $i \in [n]$, then updates the value of $\tx_{\{i\}} = x_i$ by drawing from the density
	\[
		\pi_{\{i\} \mid - \{i\}} (x \mid \tx_{- \{i\}}) = \frac{1}{1 - \sum_{j \in - \{i\}} x_j}, \quad x < 1 - \sum_{j \in - \{i\}} x_j.
	\]
	We will use Corollary~\ref{cor:corr} to construct a sharp lower bound on the spectral gap of this chain.
	We then briefly illustrates how Corollary~\ref{cor:specind} can be used to construct a similar but looser bound.
	
	\subsection{A spectral gap bound based on correlation coefficients}

	We will prove the following result for the chain in question.
	
	\begin{proposition} \label{pro:example-corr}
		Let $m \in \{2,\dots,n\}$.
		Let $\Lambda \subset [n]$ be such that $|\Lambda| = n - m$, and let $\tx = (x_1,\dots,x_n) \in \X$.
		Assume that $\sum_{i \in \Lambda} x_i < 1$.
		Then 
		\[
		s(\Lambda,\tx_{\Lambda}) \leq \begin{cases}
			3/4 & m = 2, \\
			1/m + 2(m-1)/[(m+1)m^2] & m \geq 3.
		\end{cases}
		\]
	\end{proposition}

	By Corollary~\ref{cor:corr}, when $n \geq 3$, the spectral gap satisfies
	\[
	\mbox{Gap}(n,1) \geq \frac{1}{4} \prod_{m=3}^n \left[ 1 - \frac{1}{m} - \frac{2(m-1)}{(m+1)m^2} \right] .
	\]
	Note that $1/m + 2(m-1)/[(m+1) m^2] \leq 1/(m-2)$.
	Thus, if $n \geq 4$, then
	\[
	\mbox{Gap}(n,1) \geq \frac{5}{36} \prod_{m=4}^n \frac{m-3}{m-2} = \frac{5}{36(n-2)}.
	\]
	Recall that the spectral gap is upper bounded by $1/n$.
	Thus, the bound here gives the correct order as $n \to \infty$.
	
	To prove Proposition~\ref{pro:example-corr}, fix $m \in \{2,\dots,n\}$, $\Lambda \subset [n]$ such that $|\Lambda| = n - m$, and $\tx = (x_1,\dots,x_n) \in \X$.
	Suppose that $\sum_{i \in \Lambda} x_i < 1$.
	Without loss of generality, assume that $- \Lambda = \{1,\dots,m\}$.
	Let $Y_1,\dots,Y_m$ be distributed as $\pi_{- \Lambda \mid \Lambda}(\cdot \mid \tx_{\Lambda})$.
	For $i = 1,\dots,m$, let $f_i \in L_0^2(\varpi_i)$, where $\varpi_i$ denotes the distribution given by the density 
	\begin{equation} \label{eq:example-pii}
	\pi_{\{i\} \mid \Lambda}(x \mid \tx_{\Lambda}) = \frac{m \left( 1 - \sum_{j=m+1}^n x_j - x \right)^{m-1} }{ \left( 1 - \sum_{j=m+1}^n x_j \right)^m }, \quad x < 1 - \sum_{j=m+1}^n x_j.
	\end{equation}
	It suffices to prove that
	\begin{equation} \label{eq:example-corr}
	\E \left\{ \left[ \sum_{i=1}^m f_i(Y_i) \right]^2 \right\} \leq A_m \sum_{i=1}^m \E \left[f_i(Y_i)^2 \right],
	\end{equation}
	where
	\begin{equation} \label{eq:Am}
	A_m = \begin{cases}
		3/2 & m = 2, \\
		1 + 2(m-1)/[(m+1)m] & m \geq 3.
	\end{cases}
	\end{equation}
	We will prove this using orthogonal polynomials.
	The techniques we employ are similar to those in \cite{diaconis2008gibbs}.

	Let $i,j \in - \Lambda$ be such that $i \neq j$.
	Then, for $x \in \X_i$ such that $\pi_{\{i\} \mid \Lambda}(x \mid \tx_{\Lambda}) > 0$ and $x' \in \X_j$, the conditional density of $Y_j$ given $Y_i = x$ is
	\begin{equation} \label{eq:example-piji}
	\pi_{\{j\} \mid \Lambda \cup \{i\}}(x' \mid x, \tx_{\Lambda}) = \frac{(m-1) \left( 1 - \sum_{a=m+1}^n x_a - x - x' \right)^{m-2}}{\left( 1 - \sum_{a=m+1}^n x_a - x \right)^{m-1} }, \quad x' < 1 - \sum_{a=m+1}^n x_a - x,
	\end{equation}
	where $(x,\tx_{\Lambda}) = (x, x_{m+1}, \dots, x_n)$.
	For $f \in L_0^2(\varpi_j)$, let $P_{i,j} f$ be a function on $\X_i$ such that
	\[
	P_{i,j} f(x) = \int_{\X_j} f(x') \pi_{\{j\} \mid \Lambda \cup \{i\}}(x' \mid x, \tx_{\Lambda}) \, \df x', \quad x \in \X_i.
	\]
	Using Cauchy-Schwarz, it is easy to show that $P_{i,j} f \in L_0^2(\varpi_i)$.
	In fact, $P_{i,j}: L_0^2(\varpi_j) \to L_0^2(\varpi_i)$ is a bounded linear transformation.
	Let $P_{i,i}$ be the identity on $L_0^2(\varpi_i)$.
	Then, for $i,j \in - \Lambda$,
	\begin{equation} \label{eq:pij-adjoint}
	\langle f_i, P_{i,j} f_j \rangle_{\varpi_i} = \langle P_{j,i} f_i, f_j \rangle_{\varpi_j} = \E \left[ f_i(Y_i) f_j(Y_j) \right] . 
	\end{equation}
	It follows that
	\begin{equation} \label{eq:esumfi-innerproduct}
	\begin{aligned}
		&\sum_{i=1}^m \E \left[f_i(Y_i)^2 \right] = \sum_{i=1}^m \langle f_i, P_{i,i} f_i \rangle_{\varpi_i} = \sum_{i=1}^m \|f_i\|_{\varpi_i}^2, \\
		&\E \left\{ \left[ \sum_{i=1}^m f_i(Y_i) \right]^2 \right\} = \sum_{i=1}^m \sum_{j=1}^m \langle f_i, P_{i,j} f_j \rangle_{\varpi_i}.
	\end{aligned}
	\end{equation}
	
	Now, for a positive integer~$k$ and $i,j \in - \Lambda$ such that $i \neq j$, the following holds if $\pi_{\{i\} \mid \Lambda}(x \mid \tx_{\Lambda}) > 0$:
	\begin{equation} \label{eq:xk}
	\int_{\X_j} x'^k \pi_{\{j\} \mid \Lambda \cup \{i\} }(x' \mid x, \tx_{\Lambda}) \, \df x = \zeta_k x^k + q_{k-1}(x),
	\end{equation}
	where
	\[
	\zeta_k = \frac{(-1)^k k!(m-1)!}{(m+k-1)!},
	\]
	and $q_{k-1}(x)$ is a polynomial of~$x$ whose degree is $k-1$.
	Standard arguments show that, for $i \in - \Lambda$, $L_0^2(\varpi_i)$ has an orthonormal basis $\{p_{i,k}\}_{k=1}^{\infty}$, where $p_{i,k}$ is a polynomial function of degree~$k$.
	By \eqref{eq:xk}, when $i \neq j$,
	\[
	\begin{aligned}
		P_{i,j} p_{j,k} = \zeta_k p_{i,k} + r_{i,j,k-1},
	\end{aligned}
	\]
	where $r_{i,j,k-1}$ is in the span of $\{p_{i,1},\dots,p_{i,k-1}\}$.
	Claim: for $k \geq 1$ and $i \neq j$, $r_{i,j,k} = 0$.
	This can be proved through induction.
	The claim holds for $k=1$, because the only polynomial of order~0 in $L_0^2(\varpi_i)$ is~0.
	Assume that it holds for $k= k'-1 \geq 1$.
	Then, by~\eqref{eq:pij-adjoint} and the fact that $\{p_k\}$ is an orthonormal basis, for $k = 1,\dots,k'$ and $i \neq j$,
	\[
	\begin{aligned}
		\langle r_{i,j,k'}, p_{i,k} \rangle_{\varpi_i} &= \langle P_{i,j} p_{j,k'+1} - \zeta_{k'+1} p_{i,k'+1}, p_{i,k} \rangle_{\varpi_i} \\
		&= \langle P_{i,j} p_{j,k'+1} , p_{i,k} \rangle_{\varpi_i} \\
		&= \langle p_{j,k'+1} , P_{j,i} p_{i,k} \rangle_{\varpi_j} \\
		&= \zeta_k \langle p_{j,k'+1} , p_{j,k} \rangle_{\varpi_j} \\
		&= 0.
	\end{aligned}
	\]
	This implies that $r_{i,j,k'} = 0$.
	Thus, for $k \geq 1$ and $i \neq j$, 
	\[
	P_{i,j} p_{j,k} = \zeta_k p_{i,k}.
	\]
	
	For $i \in - \Lambda$, we can decompose $f_i \in L_0^2(\varpi_i)$ into $f_i = \sum_{k=1}^{\infty} a_{i,k} p_{i,k}$.
	Then~\eqref{eq:esumfi-innerproduct} can be written as
	\[
	\begin{aligned}
		&\sum_{i=1}^m \E \left[f_i(Y_i)^2 \right] = \sum_{i=1}^m \sum_{k=1}^{\infty} a_{i,k}^2, \\
		&\E \left\{ \left[ \sum_{i=1}^m f_i(Y_i) \right]^2 \right\} = \sum_{i=1}^m \sum_{j=1}^m \sum_{k=1}^{\infty} a_{i,k} a_{j,k} [\ind_{i=j} + \ind_{i \neq j} \zeta_k].
	\end{aligned}
	\]
	Elementary matrix algebra shows that, given a positive integer~$k$,
	\[
	\sum_{i=1}^m \sum_{j=1}^m a_{i,k} a_{j,k} [\ind_{i=j} + \ind_{i \neq j} \zeta_k] \leq \max\{ 1-\zeta_k, 1 + (m-1)\zeta_k \} \sum_{i=1}^m a_{i,k}^2.
	\]
	It follows that
	\[
	\begin{aligned}
		\sum_{i=1}^m \sum_{j=1}^m \sum_{k=1}^{\infty} a_{i,k} a_{j,k} [\ind_{i=j} + \ind_{i \neq j} \zeta_k] &\leq \left( \sup_{k} \max\{ 1-\zeta_k, 1 + (m-1)\zeta_k \} \right) \sum_{k=1}^{\infty}  \sum_{i=1}^m a_{i,k}^2 \\
		&= \left[ \max\{ 1 - \zeta_1, 1 + (m-1) \zeta_2 \}  \right] \sum_{k=1}^{\infty} \sum_{i=1}^m a_{i,k}^2 \\
		& = A_m \sum_{k=1}^{\infty} \sum_{i=1}^m a_{i,k}^2,
	\end{aligned}
	\]
	where $A_m$ is given in~\eqref{eq:Am}.
	This establishes~\eqref{eq:example-corr}, and in turn, Proposition~\ref{pro:example-corr}.

	\subsection{A spectral gap bound based on influence matrices}
	
	One can also use Corollary~\ref{cor:specind} to bound the spectral gap.
	However, the bound would be looser than the one obtained in the previous subsection.
	Hence, we will not present the full calculation for this alternative bound.
	Instead, we only present parts of the calculation to illustrate how influence matrices are computed.
	
	We will establish the following result for our example.
	
	\begin{proposition} \label{pro:example-specind}
		Assume that $n \geq 4$.
		Let $m \in \{4,\dots,n\}$.
		Then, for $\Lambda \subset [n]$ such that $|\Lambda| = n - m$ and $\ty \in \X_{\Lambda}$ such that $\pi_{\Lambda}(\ty) > 0$, there is an influence matrix $\Phi(\Lambda,\ty)$ associated with $(\Lambda,\ty)$ such that $r(\Phi(\Lambda,\ty)) = (m-1)/(m-2)$.
	\end{proposition}
	
	By Corollary~\ref{cor:specind}, Proposition~\ref{pro:example-specind} implies that
	\[
	\mbox{Gap}(m,m-1) \geq \frac{m-1}{m} - \frac{(m-1)}{m(m-2)}
	\]
	for $m \geq 4$.
	If one can obtain a non-trivial lower bound $c > 0$ on $\mbox{Gap}(3,2)$ and $\mbox{Gap}(4,3)$ (which can be achieved through Corollary~\ref{cor:specind}, but requires some tedious calculations), then, by Theorem~\ref{thm:gap},
	\[
	\mbox{Gap}(n,1) \geq c^2 \prod_{m=4}^n \left[ \frac{m-1}{m} - \frac{(m-1)}{m(m-2)} \right] = \frac{3 c^2}{n(n-2)}.
	\]

	Let us now prove Proposition~\ref{pro:example-specind}.
	Assume that $n \geq 4$ and let $m \in \{4,\dots,n\}$.
	Fix $\Lambda \subset [n]$ such that $|\Lambda| = n - m$.
	Let $\ty \in \X_{\Lambda}$, and let $\tx = (x_1,\dots,x_n) \in \X$ be such that $\tx_{\Lambda} = \ty$.
	Assume that $\sum_{i \in \Lambda} x_i < 1$, so that $\pi_{\Lambda}(\ty) > 0$.
	Without loss of generality, assume that $- \Lambda = \{1,\dots,m\}$.
	
	For $i \in - \Lambda$, define a distance-like function 
	\[
	d_{\Lambda,\ty,i}(x,x') = \frac{|x-x'|}{1 - \sum_{j \in \Lambda} x_j - x \vee x'}, \quad x, x' \in \X_i,
	\]
	where $x \vee x' = \max\{x,x'\}$.
	We first need to verify that (H1) and (H2), which are given in Section~\ref{ssec:specind}, hold.
	Establishing (H1) through~\eqref{eq:example-pii} is rather straightforward.
	To establish (H2), recall that the total variation distance between two distributions equals half of the integration of the absolute difference of their density functions.
	Then, based on~\eqref{eq:example-piji}, one can find that, for $i,j \in - \Lambda$ such that $i \neq j$ and $x,x' < 1 - \sum_{a \in \Lambda} x_a$,
	\[
		\begin{aligned}
				&d_{\scriptsize\mbox{TV}}(\Pi^j_{\Lambda,\ty,i,x}, \Pi^j_{\Lambda,\ty,i,x'}) \\
				=& \frac{|x-x'|^{m-1}}{\left| \left( 1 - \sum_{a \in \Lambda} x_a - x \right)^{(m-1)/(m-2)} - \left( 1 - \sum_{a \in \Lambda} x_a - x' \right)^{(m-1)/(m-2)} \right|^{m-2}} \\
				\leq & \left( \frac{m-2}{m-1} \right)^{m-2} d_{\Lambda,\ty,i}(x,x').
			\end{aligned}
	\]
	This establishes (H2).
	
	It remains to establish a set of appropriate contraction conditions.
	For $i,j \in - \Lambda$ such that $i \neq j$, define a coupling kernel associated with $(\Lambda,\ty,i,j)$, denoted by $K_{i,j}$, as follows.
	Let~$X$ follow the distribution given by the density function
	\[
	x \mapsto (m-1) (1 - x)^{m-2}, \quad x \in (0,1).
	\]
	For $x,x' \in \X_i$, let $K_{i,j}((x,x'),\cdot)$ be the distribution of 
	\[
	\left( \left(1 - \sum_{a \in \Lambda} x_a - x \right)X, \left(1 - \sum_{a \in \Lambda} x_a - x' \right)X \right).
	\]
	One can verify that this is a valid coupling kernel.
	In particular, $K_{i,j}((x,x'),\cdot)$ is a coupling of $\Pi^j_{\Lambda,\ty,i,x}$ and $\Pi^j_{\Lambda,\ty,i,x'}$.
	Now, 
	\[
	\begin{aligned}
		&\int_{\X_j \times \X_j} \int_{\X_j \times \X_j} d_{\Lambda,\ty,j}(x'', x''') \, K_{i,j} \left((x,x'), \df (x'',x''') \right)  \\
		=& \E \left[ \frac{|x-x'| X}{ 1 - \sum_{a \in \Lambda} x_a - \left( 1 - \sum_{a \in \Lambda} x_a - x \wedge x' \right) X } \right] \\
		\leq & \frac{|x-x'|}{1 - \sum_{a \in \Lambda} x_a - x \wedge x' } \, \E \left( \frac{X}{1 - X} \right) \\
		\leq & \frac{d_{\Lambda,\ty,i}(x, x')}{m-2} ,
	\end{aligned}
	\]
	where $x \wedge x' = \min\{x,x'\}$.
	The above calculation shows that there exists an influence matrix $\Phi(\Lambda,\ty)$ associated with $(\Lambda,\ty)$ whose non-diagonal elements are $1/(m-2)$.
	Then $r(\Phi(\Lambda,\ty)) = (m-1)/(m-2)$.
	This proves Proposition~\ref{pro:example-specind}.
	
	If one instead use the discrete metric to construct the influence matrix $\Phi(\Lambda,\ty)$, then all the off-diagonal entries of $\Phi(\Lambda,\ty)$ would be~1.
	The spectral gap obtained through Corollary~\ref{cor:specind} would then be trivial.
	
	\subsection{Discussion}
	
	We see from this example that both Corollaries~\ref{cor:corr} and~\ref{cor:specind} are capable of giving reasonably sharp bounds on the spectral gap.
	However, to construct these bounds, we need sufficient information on $\pi_{\{j\} \mid \Lambda \cup \{i\}}$ for every $\Lambda \subset [n]$ such that $|\Lambda| \in \{0,\dots,n-2\}$ and $i,j \in - \Lambda$ such that $i \neq j$.
	For many practical problems, $\pi_{\{j\} \mid \Lambda \cup \{i\}}$ is intractable, especially when $\Lambda \cup \{i\} \cup \{j\} \neq [n]$.
	Indeed, even for chains on finite state spaces, spectral independence is often non-trivial to establish.
	A subject for future research would be to apply spectral telescope to analyze Gibbs chains used in various fields. Our results may be useful to study certain  physics models, similar to those studied in \cite{janvresse2001spectral}, \cite{carlen2003determination}, \cite{johnson2015geometric}, and \cite{pillai2018mixing}; and statistical models such as the de Finitti's priors for almost exchangeable data \citep{gerencser2019mixing,gerencser2020rates}.

	\section{Derivation of Main Results} \label{sec:derivation}

	\subsection{Hierarchical Structure of the Spectral Gap}
	
	In this subsection, we derive Theorem~\ref{thm:gap}.
	It suffices to prove the following lemma, which, as we will see, follows from the recursive representation of Algorithm~\ref{alg:blo-gibbs} given in Algorithm~\ref{alg:rec-gibbs}.
	\begin{lemma} \label{lem:induction}
		Let $l \in \{1,\dots,n-1\}$ and $m \in \{l+1,\dots,n\}$.
		Let $\Lambda \subset [n]$ be such that $|\Lambda| = n - m$, and let $\ty \in \X_{\Lambda}$ be such that $\pi_{\Lambda}(\ty) > 0$.
		Then
		\[
		\mbox{gap}  (\Lambda,\ty,l) \geq \mbox{Gap}  (m,m-1)  \mbox{Gap} (m-1,l).
		\]
	\end{lemma}
	
	To begin our analysis, fix $l \in \{1,\dots,n-1\}$, $\Lambda \subset [n]$ such that $|\Lambda| = n-m$ where $m \in \{l+1,\dots,n\}$, and $\ty \in \X_{\Lambda}$ where $\pi_{\Lambda}(\ty) > 0$.
	Denote by~$\varpi$ the probability measure given by $\pi_{- \Lambda \mid \Lambda}(\cdot \mid \ty)$.
	For $i \in - \Lambda$ and $x \in \X_i$, let $\varpi_{i,x}$ be the probability measure given by $\pi_{ - (\Lambda \cup \{i\}) \mid \Lambda \cup \{i\}}(\cdot \mid \tx_{\Lambda \cup \{i\}})$ where $\tx \in \X$ satisfies $\tx_{\Lambda} = \ty$ and $\tx_{\{i\}} = x$.

	Denote the Mtk of Algorithm~\ref{alg:blo-gibbs} targeting~$\varpi$ with block size~$l$ by  $K(\cdot, \cdot)$.
	Then, for $f \in L^2(\varpi)$ and $\tx \in \X$,
	\[
	Kf(\tx_{- \Lambda}) = \frac{1}{{m \choose l}} \sum_{\stackrel{\Gamma \subset - \Lambda}{|\Gamma| = l}} \E \left[ f(\tX_{- \Lambda}) \mid \tX_{- (\Lambda \cup \Gamma)} = \tx_{- (\Lambda \cup \Gamma)}, \tX_{\Lambda} = \ty \right].
	\]
	It is straightforward to check that~$K$ defines a positive semi-definite operator on $L_0^2(\varpi)$.
	Its spectral gap is $1 - \|K\|_{\varpi} = \mbox{gap}(\Lambda,\ty,l)$.
	
	Denote the Mtk of Algorithm~\ref{alg:blo-gibbs} targeting~$\varpi$ with block size $m-1$ by $\bar{K}(\cdot,\cdot)$.
	Then, for $f \in L^2(\varpi)$ and $\tx \in \X$,
	\[
	\bar{K}f(\tx_{- \Lambda}) = \frac{1}{m} \sum_{i \in - \Lambda} \E \left[ f(\tX_{- \Lambda}) \mid X_i = \tx_{\{i\}}, \tX_{\Lambda} = \ty \right].
	\]
	$\bar{K}$ defines a positive semi-definite operator on $L_0^2(\varpi)$, and its spectral gap satisfies
	\begin{equation} \label{ine:gap-Kbar}
		1 - \|\bar{K}\|_{\varpi} \geq \mbox{Gap}(m,m-1).
	\end{equation}
	
	Denote the Mtk of Algorithm~\ref{alg:blo-gibbs} targeting $\varpi_{i,x}$ with block size~$l$, where $i \in - \Lambda$ and $x \in \X_i$, by $K_{i,x}(\cdot, \cdot)$.
	Then, for $f \in L^2(\varpi_{i,x})$ and $\tx \in \X$,
	\[
	\begin{aligned}
		&K_{i,x}f(\tx_{ - (\Lambda \cup \{i\}) } ) \\
		= & \frac{1}{{m-1 \choose l}} \sum_{\stackrel{\Gamma \subset - (\Lambda \cup \{i\})}{ |\Gamma| = l }} \E \left[ f(\tX_{- (\Lambda \cup \{i\})}) \mid \tX_{- (\Lambda \cup \Gamma \cup \{i\})} = \tx_{- (\Lambda \cup \Gamma \cup \{i\})}, X_i = x, \tX_{\Lambda} = \ty \right].
	\end{aligned}
	\]
	One can check that, for $\pi_{\{i\} \mid \Lambda}(\cdot \mid \ty)$-almost every $x \in \X_i$, $K_{i,x}$ defines a positive semi-definite operator on $L_0^2(\varpi_{i,x})$, and its spectral gap satisfies
	\begin{equation} \label{ine:gap-Kix}
		1 - \|K_{i,x}\|_{\varpi_{i,x}} \geq \mbox{Gap}(m-1,l).
	\end{equation}

	For $f \in L^2(\varpi)$, $i \in - \Lambda$, and $x \in \X_i$, let $f_{i,x}: \X_{- (\Lambda \cup \{i\})} \to \mathbb{R}$ be such that $f(\tx_{- \Lambda}) = f_{i,x}(\tx_{- (\Lambda \cup \{i\})})$ whenever $\tx_{\{i\}} = x$.
	In other words, $f_{i,x}$ is just~$f$ with the $\X_i$-component of its argument fixed at~$x$.
	For instance, if~$f$ is a function on $\X_1 \times \X_2$, then $f_{1,x}(x_2) = f(x,x_2)$ for $x \in \X_1$ and $x_2 \in \X_2$, whereas $f_{2,x}(x_1) = f(x_1, x)$ for $x_1 \in \X_1$ and $x \in \X_2$.
	Given $f \in L^2(\varpi)$, for $\pi_{\{i\} \mid \Lambda}(\cdot \mid \ty)$-almost every $x \in \X_i$, $f_{i,x} \in L^2(\varpi_{i,x})$.
	For $f \in L^2(\varpi)$, the following holds $\varpi$-almost everywhere on $\X_{-\Lambda}$:
	\begin{equation} \label{eq:Kf-rec}
		\begin{aligned}
			& \frac{1}{m} \sum_{i \in - \Lambda} K_{i, \tx_{\{i\}}} f_{i, \tx_{\{i\}} } (\tx_{ - (\Lambda \cup \{i\}) }) \\
			=& \frac{1}{m} \sum_{i \in - \Lambda} \frac{1}{{m-1 \choose l}} \sum_{\stackrel{\Gamma \subset - (\Lambda \cup \{i\})}{ |\Gamma| = l }} \E \left[ f_{i, \tx_{\{i\}}}(\tX_{- (\Lambda \cup \{i\})}) \mid \tX_{- (\Lambda \cup \Gamma \cup \{i\})} = \tx_{- (\Lambda \cup \Gamma \cup \{i\})}, X_i = \tx_{\{i\}}, \tX_{\Lambda} = \ty \right] \\
			=& \frac{1}{m} \sum_{i \in - \Lambda} \frac{1}{{m-1 \choose l}} \sum_{\stackrel{\Gamma \subset - (\Lambda \cup \{i\})}{ |\Gamma| = l }} \E \left[ f_{i, X_i}(\tX_{- (\Lambda \cup \{i\})}) \mid \tX_{- (\Lambda \cup \Gamma \cup \{i\})} = \tx_{- (\Lambda \cup \Gamma \cup \{i\})}, X_i = \tx_{\{i\}}, \tX_{\Lambda} = \ty \right] \\
			=& \frac{1}{m} \sum_{i \in - \Lambda} \frac{1}{{m-1 \choose l}} \sum_{\stackrel{\Gamma \subset - (\Lambda \cup \{i\})}{ |\Gamma| = l }} \E \left[ f(\tX_{-\Lambda}) \mid \tX_{- (\Lambda \cup \Gamma \cup \{i\})} = \tx_{- (\Lambda \cup \Gamma \cup \{i\})}, X_i = \tx_{\{i\}}, \tX_{\Lambda} = \ty \right] \\
			=& \frac{1}{m} \sum_{i \in - \Lambda} \frac{1}{{m-1 \choose l}} \sum_{\stackrel{\Gamma \subset - (\Lambda \cup \{i\})}{ |\Gamma| = l }} \E \left[ f(\tX_{- \Lambda}) \mid \tX_{- (\Lambda \cup \Gamma )} = \tx_{- (\Lambda \cup \Gamma )}, \tX_{\Lambda} = \ty \right] \\
			=& \frac{1}{m} \frac{1}{{m-1 \choose l}} (m-l) \sum_{\stackrel{\Gamma \subset - \Lambda}{ |\Gamma| = l }} \E \left[ f(\tX_{- \Lambda}) \mid \tX_{- (\Lambda \cup \Gamma )} = \tx_{- (\Lambda \cup \Gamma )}, \tX_{\Lambda} = \ty \right] \\
			=& Kf(\tx_{- \Lambda}).
		\end{aligned}
	\end{equation}
	This formula gives precisely the equivalence between Algorithms~\ref{alg:blo-gibbs} and~\ref{alg:rec-gibbs}.
	
	To prove Lemma~\ref{lem:induction}, fix $f \in L_0^2(\varpi)$.
	For $i \in - \Lambda$, let $\Delta_i f = f - P_i f$, where $P_i f \in L_0^2(\varpi)$ satisfies
	\[
	P_i f(\tx_{- \Lambda}) = \E \left[ f(\tX_{- \Lambda}) \mid X_i = \tx_{\{i\}}, \tX_{\Lambda} = \ty \right].
	\]
	Then, for $\pi_{\{i\} \mid \Lambda}(\cdot \mid \ty)$-almost every $x \in \X_i$, $(P_i f)_{i,x} \in L^2(\varpi_{i,x})$, and $(\Delta_i f)_{i,x} \in L_0^2(\varpi_{i,x})$.
	(In fact, $f_{i,x} \mapsto (P_i f)_{i,x}$ corresponds to the orthogonal projection on $L^2(\varpi_{i,x})$ associated with the subspace of constant functions.)
	By the tower property of conditional expectations,
	\begin{equation} \label{eq:fKf}
	\begin{aligned}
		& \langle f, K f \rangle_{\varpi} \\
		=& \E \left[ f(\tX_{- \Lambda}) K f(\tX_{- \Lambda}) \mid \tX_{\Lambda} = \ty \right] \\
		=& \frac{1}{m} \sum_{i \in - \Lambda} \E \left[ f(\tX_{- \Lambda}) \, K_{i,X_i} f_{i, X_i}(\tX_{- (\Lambda \cup \{i\})}) \mid \tX_{\Lambda} = \ty \right] \quad \text{(by } \eqref{eq:Kf-rec}) \\
		=& \frac{1}{m} \sum_{i \in - \Lambda} \E \left[ f_{i, X_i}(\tX_{- (\Lambda \cup \{i\})}) \, K_{i,X_i} f_{i, X_i}(\tX_{- (\Lambda \cup \{i\})}) \mid \tX_{\Lambda} = \ty \right] \quad (\text{by definition of } f_{i,X_i} )\\
		=& \frac{1}{m} \sum_{i \in - \Lambda} \int_{\X_i} \E \left[ f_{i, x }(\tX_{- (\Lambda \cup \{i\})}) \, K_{i,x} f_{i, x}(\tX_{- (\Lambda \cup \{i\})}) \mid X_i = x, \tX_{\Lambda} = \ty \right] \pi_{\{i\} \mid \Lambda}(x \mid \ty) \,\df x \\
		=& \frac{1}{m} \sum_{i \in - \Lambda} \int_{\X_i} \langle (P_i f)_{i,x} + (\Delta_i f)_{i,x} , K_{i,x} [(P_i f)_{i,x} + (\Delta_i f)_{i,x}] \rangle_{\varpi_{i,x}} \pi_{\{i\} \mid \Lambda}(x \mid \ty) \, \df x.
	\end{aligned}
	\end{equation}
	The integrand in the last line equals
	\begin{equation} \label{eq:three-terms}
	\begin{aligned}
		& \langle (P_i f)_{i,x} + (\Delta_i f)_{i,x} , (P_i f)_{i,x} + K_{i,x} (\Delta_i f)_{i,x} \rangle_{\varpi_{i,x}} \\
		=& \langle (P_i f)_{i,x}, (P_i f)_{i,x} \rangle_{\varpi_{i,x}} + \langle (\Delta_i f)_{i,x}, K_{i,x} (\Delta_i f)_{i,x} \rangle_{\varpi_{i,x}} + \\
		& \langle (P_i f)_{i,x}, K_{i,x} (\Delta_i f)_{i,x} \rangle_{\varpi_{i,x}} + \langle (\Delta_i f)_{i,x}, (P_i f)_{i,x} \rangle_{\varpi_{i,x}} \\
		=& \langle (P_i f)_{i,x}, (P_i f)_{i,x} \rangle_{\varpi_{i,x}} + \langle (\Delta_i f)_{i,x}, K_{i,x} (\Delta_i f)_{i,x} \rangle_{\varpi_{i,x}} + 2\langle (\Delta_i f)_{i,x}, (P_i f)_{i,x} \rangle_{\varpi_{i,x}},
	\end{aligned}
	\end{equation}
	where the last equality follows from the fact that $K_{i,x}$ is self-adjoint and that $K_{i,x} (P_if)_{i,x} = (P_if)_{i,x}$.
	
	Let us examine the three terms in the last line of~\eqref{eq:three-terms}.
	Firstly, one can verify that, for $\pi_{ \{i\} \mid \Lambda}(\cdot \mid \ty)$-almost every $x \in \X_i$,
	\begin{equation} \label{eq:PifPif-ix}
	\begin{aligned}
		\langle (P_i f)_{i,x}, (P_i f)_{i,x} \rangle_{\varpi_{i,x}} &=  \E \left\{ f(\tX_{- \Lambda}) \, \E \left[ f(\tX_{- \Lambda}) \mid X_i = x, \tX_{\Lambda} = \ty \right] \mid X_i = x, \tX_{\Lambda} = \ty \right\} \\
		&= \langle f_{i,x}, (P_i f)_{i,x} \rangle_{\varpi_{i,x}}.
	\end{aligned}
	\end{equation}
	It follows that
	\begin{equation} \label{eq:PifPif}
	\begin{aligned}
		&\frac{1}{m} \sum_{i \in - \Lambda} \int_{\X_i} \langle (P_i f)_{i,x}, (P_i f)_{i,x} \rangle_{\varpi_{i,x}} \, \pi_{\{i\} \mid \Lambda}(x \mid \ty) \, \df x \\
		=& \frac{1}{m} \sum_{i \in - \Lambda} \int_{\X_i} \langle f_{i,x}, (P_i f)_{i,x} \rangle_{\varpi_{i,x}} \, \pi_{\{i\} \mid \Lambda}(x \mid \ty) \, \df x \\
		=& \frac{1}{m} \sum_{i \in - \Lambda} \langle f, P_if \rangle_{\varpi} \\
		=& \langle f, \bar{K} f \rangle_{\varpi} .
	\end{aligned}
	\end{equation}
	
	Secondly, by~\eqref{ine:gap-Kix}, for $\pi_{ \{i\} \mid \Lambda}(\cdot \mid \ty)$-almost every $x \in \X_i$, since $(\Delta_i f)_{i,x} \in L_0^2(\varpi_{i,x})$,
	\[
	\begin{aligned}
		\langle (\Delta_i f)_{i,x}, K_{i,x} (\Delta_i f)_{i,x} \rangle_{\varpi_{i,x}}  
		\leq&  \|K_{i,x}\|_{\varpi_{i,x}} \| (\Delta_i f)_{i,x} \|_{\varpi_{i,x}}^2 \\
		\leq& [1 - \mbox{Gap}(m-1,l)] \| (\Delta_i f)_{i,x} \|_{\varpi_{i,x}}^2.
	\end{aligned}
	\] 
	By~\eqref{eq:PifPif-ix},
	\[
	\begin{aligned}
		&\| (\Delta_i f)_{i,x} \|_{\varpi_{i,x}}^2 \\
		=& \langle f_{i,x}, f_{i,x} \rangle_{\varpi_{i,x}} + \langle (P_i f)_{i,x}, (P_i f)_{i,x} \rangle_{\varpi_{i,x}} - 2 \langle f_{i,x}, (P_i f)_{i,x} \rangle_{\varpi_{i,x}} \\
		=& \| f_{i,x} \|_{\varpi_{i,x}}^2 - \langle f_{i,x}, (P_i f)_{i,x} \rangle_{\varpi_{i,x}}.
	\end{aligned}
	\]
	Therefore,
	\begin{equation} \label{eq:DifDif}
	\begin{aligned}
		&\frac{1}{m} \sum_{i \in - \Lambda} \int_{\X_i} \langle (\Delta_i f)_{i,x}, K_{i,x} (\Delta_i f)_{i,x} \rangle_{\varpi_{i,x}} \pi_{\{i\} \mid \Lambda} (x \mid \ty) \, \df x \\
		\leq & [1 - \mbox{Gap} (m-1,l)] \left[ \|f\|_{\varpi}^2 - \frac{1}{m} \sum_{i \in - \Lambda} \int_{\X_i} \langle f_{i,x}, (P_i f)_{i,x} \rangle_{\varpi_{i,x}} \, \pi_{\{i\} \mid \Lambda}(x \mid \ty) \, \df x \right] \\
		=& [1 - \mbox{Gap} (m-1,l)] \, ( \|f\|_{\varpi}^2 - \langle f, \bar{K} f \rangle_{\varpi} ),
	\end{aligned}
	\end{equation}
	where the final equality follows from~\eqref{eq:PifPif}.
	
	Finally, by~\eqref{eq:PifPif-ix}, for $\pi_{ \{i\} \mid \Lambda}(\cdot \mid \ty)$-almost every $x \in \X_i$,
	\[
	\langle (\Delta_i f)_{i,x}, (P_i f)_{i,x} \rangle_{\varpi_{i,x}} = \langle f_{i,x}, (P_i f)_{i,x} \rangle_{\varpi_{i,x}} - \langle (P_i f)_{i,x}, (P_i f)_{i,x} \rangle_{\varpi_{i,x}} = 0,
	\]
	so
	\begin{equation} \label{eq:DifPif}
	\frac{1}{m} \sum_{i \in - \Lambda} \int_{\X_i} \langle (\Delta_i f)_{i,x}, (P_i f)_{i,x} \rangle_{\varpi_{i,x}} \, \pi_{ \{i\} \mid \Lambda }(x \mid \ty) \, \df x = 0.
	\end{equation}
	Combining~\eqref{ine:gap-Kbar} and~\eqref{eq:fKf} to~\eqref{eq:DifPif} shows that
	\[
	\begin{aligned}
		\langle f, K f \rangle_{\varpi} \leq& [1 - \mbox{Gap}  (m-1,l)] \|f\|_{\varpi}^2 + \mbox{Gap}  (m-1,l) \langle f, \bar{K} f \rangle_{\varpi} \\
		\leq & [1 - \mbox{Gap}  (m-1,l)] \|f\|_{\varpi}^2 + \mbox{Gap} (m-1,l) [1 - \mbox{Gap}  (m,m-1)] \|f\|_{\varpi}^2 \\
		= & [1 - \mbox{Gap}  (m-1,l) \, \mbox{Gap} (m,m-1)] \|f\|_{\varpi}^2.
	\end{aligned}
	\]
	Since~$K$ is positive semi-definite, and $f \in L_0^2(\varpi)$ is arbitrary, Lemma~\ref{lem:induction} holds.
	
	\subsection{Spectral gap and correlation coefficients}
	
	In this subsection, we derive Corollary~\ref{cor:corr}.
	It suffices to show the following.
	\begin{lemma} \label{lem:corr}
		Let $m \in \{2,\dots,n\}$.
		Then, for $\Lambda \subset [n]$ such that $|\Lambda| = n-m$ and $\ty \in \X_{\Lambda}$ such that $\pi_{\Lambda}(\ty) > 0$,
		\[
		\mbox{gap}(\Lambda, \ty, m-1) \geq 1 - s(\Lambda, \ty).
		\]
		In particular,
		\[
		\mbox{Gap}(m,m-1) \geq 1 - S(m).
		\]
	\end{lemma}
	
	To prove the lemma, fix $\Lambda \subset [n]$ such that $|\Lambda| = n-m$ where $m \in \{2,\dots.n\}$ and $\ty \in \X_{\Lambda}$ such that $\pi_{\Lambda}(\ty) > 0$.
	As in the previous section, denote by~$\varpi$ the probability measure given by $\pi_{- \Lambda \mid \Lambda}(\cdot \mid \ty)$, and let $\bar{K}(\cdot,\cdot)$ be the Mtk of Algorithm~\ref{alg:blo-gibbs} targeting~$\varpi$ with block size $m-1$.
	Then
	\[
	\mbox{gap}(\Lambda,\ty,m-1) = 1 - \|\bar{K}\|_{\varpi}.
	\]
	For $f \in L_0^2(\varpi)$,
	\[
	\bar{K} f = \frac{1}{m} \sum_{i \in - \Lambda} P_i f,
	\]
	where $P_i: L_0^2(\varpi) \to L_0^2(\varpi)$ satisfies
	\[
	P_i f(\tx_{- \Lambda}) = \E \left[ f(\tX_{- \Lambda}) \mid X_i = \tx_{\{i\}}, \tX_{\Lambda} = \ty \right].
	\]
	For $i \in - \Lambda$, $P_i^2 = P_i$, and for $f, g \in L_0^2(\varpi)$,
	\[
	\langle P_i f, g \rangle_{\varpi} = \langle P_i f, P_i g \rangle_{\varpi} = \langle f, P_i g \rangle_{\varpi} .
	\]
	In fact, $P_i$ is the orthogonal projection onto the space of $L_0^2(\varpi)$ functions $\tx_{- \Lambda} \mapsto f(\tx_{- \Lambda})$ that depend on $\tx_{- \Lambda}$ only through $\tx_{\{i\}}$.
	Denote the range of $P_i$ by $\Range_i$.
	Let $\Range = \sum_{i \in - \Lambda} \Range_i$.
	That is, $\Range \subset L_0^2(\varpi)$ consists of functions that are sums of functions from $\Range_i$.
	Obviously, $\bar{K}$ maps a function in $L_0^2(\varpi)$ to a function in~$\Range$.
	Let $\bar{K}|_{\Range}$ be~$\bar{K}$ restricted to~$\Range$.
	We then have the following lemma.
	\begin{lemma} \label{lem:Knorm}
		\begin{equation} \label{eq:Knorm}
		\|\bar{K}\|_{\varpi} = \|\bar{K}|_{\Range}\|_{\varpi}.
		\end{equation}
	\end{lemma}
	\begin{proof}
		It is clear that $\|\bar{K}\|_{\varpi} \geq \|\bar{K}|_{\Range}\|_{\varpi}$.
		It remains to prove the reverse inequality.
		Since~$\bar{K}$ is self-adjoint, its norm equals its spectral radius.
		Then, by Gelfand's formula,
		\begin{equation} \label{eq:gelfand}
		\|\bar{K}\|_{\varpi} = \lim_{t \to \infty} \|\bar{K}^t \|_{\varpi}^{1/t}.
		\end{equation}
		For any $f \in L_0^2(\varpi)$ and positive integer~$t$ such that $t \geq 2$,
		\[
		\|\bar{K}^tf\|_{\varpi} = \|\bar{K}|_{\Range}^{t-1} \bar{K} f \|_{\varpi} \leq \|\bar{K}|_{\Range}\|_{\varpi}^{t-1} \| \bar{K} \|_{\varpi} \|f\|_{\varpi}.
		\]
		Then
		\[
		\lim_{t \to \infty} \|\bar{K}^t \|_{\varpi}^{1/t} \leq \lim_{t \to \infty} \|\bar{K}|_{\Range}\|_{\varpi}^{(t-1)/t} \|\bar{K}\|_{\varpi}^{1/t} = \|\bar{K}|_{\Range}\|_{\varpi}.
		\]
		It then follows from~\eqref{eq:gelfand} that
		\[
		\|\bar{K}\|_{\varpi} \leq \|\bar{K}|_{\Range}\|_{\varpi}.
		\]
	\end{proof}
	
	To derive Lemma~\ref{lem:corr}, we combine Lemma~\ref{lem:Knorm} with a simple result from \cite{bjorstad1991spectra} concerning norms of sums of orthogonal projections.
	\begin{lemma} \citep[][Theorem 3.2]{bjorstad1991spectra} \label{lem:bjorstad}
		\[
		\|\bar{K}|_{\Range}\|_{\varpi} \leq \sup_{\stackrel{f_i \in \Range_i \; \forall i}{\exists i \text{ s.t. } f_i \neq 0}} \frac{\left\| \sum_{i \in - \Lambda} f_i \right\|_{\varpi}^2 }{ m \sum_{i \in - \Lambda} \|f_i\|_{\varpi}^2 }.
		\]
	\end{lemma}

	Note that for $f \in L_0^2(\varpi)$, 
	\[
	\|f\|_{\varpi}^2 = \E \left[ f(\tX_{- \Lambda})^2 \mid \tX_{\Lambda} = \ty \right].
	\]
	Moreover, if we let $\varpi_i$ be the probability measure on $(\X_i,\B_i)$ given by $\pi_{\{i\} \mid \Lambda}(\cdot \mid \ty)$, then there is a natural isomorphism from $\Range_i$ to $L_0^2(\varpi_i)$.
	It follows that
	\[
	\begin{aligned}
		\sup_{\stackrel{f_i \in \Range_i \; \forall i}{\exists i \text{ s.t. } f_i \neq 0}} \frac{\left\| \sum_{i \in - \Lambda} f_i \right\|_{\varpi}^2 }{ m \sum_{i \in - \Lambda} \|f_i\|_{\varpi}^2 } 
		=& \sup_{\stackrel{f_i \in L_0^2(\varpi_i) \; \forall i}{\exists i \text{ s.t. } \E[f_i(X_i)^2 \mid \tX_{\Lambda} = \ty] > 0}} \frac{\E \left\{ \left[ \sum_{i \in - \Lambda} f_i(X_i) \right]^2 \mid \tX_{\Lambda} = \ty \right\}}{m \sum_{i \in - \Lambda} \E [ f_i(X_i)^2 \mid \tX_{\Lambda} = \ty ] } \\
		=& s(\Lambda, \ty).
	\end{aligned}
	\]
	Lemma~\ref{lem:corr} then follows from Lemmas~\ref{lem:Knorm} and~\ref{lem:bjorstad}.
	
	\subsection{Spectral gap and random walks} \label{ssec:proof-randomwalk}
	
	In this section, we derive Corollary~\ref{cor:randomwalk}.
	In light of Corollary~\ref{cor:corr}, it suffices to prove the following result.
	
	\begin{lemma} \label{lem:rw}
		Let $m \in \{2,\dots,n\}$.
		Then, for $\Lambda \subset [n]$ such that $|\Lambda| = n-m$ and $\ty \in \X_{\Lambda}$ such that $\pi_{\Lambda}(\ty) > 0$,
		\[
		g(\Lambda,\ty) = 1 - s(\Lambda, \ty).
		\]
		In particular,
		\[
		G(m) = 1 - S(m).
		\]
	\end{lemma}

	To prove Lemma~\ref{lem:rw}, fix $\Lambda \subset [n]$ such that $|\Lambda| = n-m$, where $m \in \{2,\dots,n\}$, and $\ty \in \X_{\Lambda}$ such that $\pi_{\Lambda}(\ty) > 0$.
	
	Consider Algorithm~\ref{alg:randomwalk} associated with~$\Lambda$ and~$\ty$.
	Recall that the underlying random walk Markov chain has $\tilde{\X} = \bigcup_{i \in - \Lambda} (\{i\} \times \X_i)$ as its state space.
	The chain is reversible with respect to the probability measure~$\varphi$ given by
	\[
	\varphi(\{i\} \times A) = \frac{1}{m} \varpi_i(A), \quad i \in - \Lambda, \; A \in \B_i,
	\]
	where $\varpi_i$ is the probability measure on $(\X_i,\B_i)$ given by $\pi_{\{i\} \mid \Lambda}(\cdot \mid \ty)$.
	A measurable function on~$\tilde{\X}$ has the form $(i,x) \mapsto f(i,x)$, where $i \in - \Lambda$ and $x \in \X_i$.
	For such a function~$f$, we can identify~$m$ functions $(T_i f)_{i \in - \Lambda}$, such that $T_i f(x) = f(i,x)$ for $i \in - \Lambda$ and $x \in \X_i$.
	Then $f \in L_0^2(\varphi)$ if and only if $T_i f \in L^2(\varpi_i)$ for each~$i$, and
	\begin{equation} \label{eq:Tif-sum}
	\sum_{i \in - \Lambda} \E [T_i f(X_i) \mid \tX_{\Lambda} = \ty] = 0.
	\end{equation}
	For $f \in L_0^2(\varphi)$,
	\begin{equation} \label{eq:fnorm-varphi}
	\|f\|_{\varphi}^2 = \frac{1}{m} \sum_{i \in - \Lambda} \E\{ [T_if(X_i)]^2 \mid \tX_{\Lambda} = \ty \} .
	\end{equation}

	Let $\RW(\cdot,\cdot)$ be the Mtk of Algorithm~\ref{alg:randomwalk} associated with~$\Lambda$ and~$\ty$.
	$\RW(\cdot,\cdot)$ defines the following operator on $L_0^2(\varphi)$:
	For $f \in L_0^2(\varphi)$, $i \in - \Lambda$, and $x \in \X_i$,
	\begin{equation} \label{eq:RWf}
		\RW f(i,x) = \frac{1}{m} \sum_{j \in - \Lambda} \E[ T_j f(X_j) \mid X_i = x, \, \tX_{\Lambda} = \ty] .
	\end{equation}
	It follows that, for $f \in L_0^2(\varphi)$,
	\begin{equation} \label{eq:fRWf}
	\begin{aligned}
		\langle f, \RW f \rangle_{\varphi} =& \frac{1}{m^2} \sum_{i,j \in - \Lambda} \E[T_i f(X_i) \, T_j f(X_j) \mid \tX_{\Lambda} = \ty] \\
		=& \frac{1}{m^2} \E \left\{ \left[\sum_{i \in - \Lambda} T_i f(X_i) \right]^2 \Big | \; \tX_{\Lambda} = \ty \right\} .
	\end{aligned}
	\end{equation}
	From this formula, we can see that~$\RW$ is positive semi-definite.
	
	Let $\Range'$ be the space of functions~$f$ in $L_0^2(\varphi)$ such that 
	\[
	\E [T_i f(X_i) \mid \tX_{\Lambda} = \ty] = 0
	\] 
	for $i \in - \Lambda$.
	In other words, $f \in \Range'$ if and only if $T_i f \in L_0^2(\varpi_i)$ for $i \in - \Lambda$.
	By~\eqref{eq:Tif-sum} and~\eqref{eq:RWf}, for $f \in L_0^2(\varphi)$, $\RW f \in \Range'$.
	Just like in Lemma~\ref{lem:Knorm}, one can argue that
	\[
	1 - g(\Lambda, \ty) = \|\RW\|_{\varphi} = \|\RW|_{\Range'}\|_{\varphi} ,
	\]
	where $\RW|_{\Range'}$ is~$\RW$ restricted to~$\Range'$.
	It then follows from~\eqref{eq:fnorm-varphi},~\eqref{eq:fRWf}, and the fact that~$\RW$ is positive semi-definite that 
	\begin{equation} \label{eq:1-g}
	1 - g(\Lambda, \ty) = \sup_{\stackrel{f \in \Range'}{f \neq 0}} \frac{\langle f, \RW f \rangle_{\varphi}}{\|f\|_{\varphi}^2} \leq s(\Lambda, \ty).
	\end{equation}
	
	It remains to show the reverse inequality.
	To this end, let $(f_i)_{i \in - \Lambda}$ be such that $f_i \in L_0^2(\varpi_i)$ for each~$i$, and that $f_i \neq 0$ for some~$i$.
	One can find a function $f \in L_0^2(\varphi)$ such that
	\[
	f(i,x) = T_i f(x) = f_i(x)
	\] 
	for $i \in - \Lambda$ and $x \in \X_i$.
	Then, by~\eqref{eq:fnorm-varphi} and~\eqref{eq:fRWf},
	\[
	\begin{aligned}
		\frac{\E \left\{ \left[ \sum_{i \in - \Lambda} f_i(X_i) \right]^2 \mid \tX_{\Lambda} = \ty \right\}}{m \sum_{i \in - \Lambda} \E [ f_i(X_i)^2 \mid \tX_{\Lambda} = \ty ] } &= \frac{\E \left\{ \left[ \sum_{i \in - \Lambda} T_i f(X_i) \right]^2 \mid \tX_{\Lambda} = \ty \right\}}{m \sum_{i \in - \Lambda} \E [ T_i f(X_i)^2 \mid \tX_{\Lambda} = \ty ] } \\
		& = \frac{\langle f, \RW f \rangle_{\varphi}}{\|f\|_{\varphi}^2} \\
		& \leq \|\RW\|_{\varphi}.
	\end{aligned}
	\]
	This shows that
	\[
	s(\Lambda, \ty) \leq 1 - g(\Lambda, \ty).
	\]
	In summary, Lemma~\ref{lem:rw} holds.

	\subsection{Spectral gap and spectral independence} \label{ssec:specind-proof}
	
	In this section, we prove Corollary~\ref{cor:specind}.
	In light of Corollary~\ref{cor:randomwalk}, it suffices to prove the following.
	\begin{lemma} \label{lem:specind}
		Let $\Lambda \subset [n]$ be such that $|\Lambda| = n-m$, where $m \in \{2,\dots,n\}$, and let $\ty \in \X_{\Lambda}$ be such that $\pi_{\Lambda}(\ty) > 0$.
		Suppose that there is an influence matrix $\Phi(\Lambda,\ty)$ associated with $(\Lambda,\ty)$ such that 
		\[
		r(\Phi(\Lambda,\ty)) \leq \eta,
		\]
		where $\eta < m-1$.
		Then
		\[
		g(\Lambda,\ty) \geq \frac{m-1}{m} - \frac{\eta}{m}.
		\]
	\end{lemma}	
	
	The proof is divided into several steps.
	We first define an altered version of the random walk and relate the $L^2$ norm of its Markov operator to $g(\Lambda,\ty)$, the spectral gap of the original random walk.
	We then incorporate a coupling argument, somewhat similar to that used in \cite{feng2021rapid}, to construct a convergence bound for the altered random walk in a Wasserstein divergence.
	Next, we use one-shot coupling \citep{roberts2002one,madras2010quantitative} to translate the bound to one in total variation distance.
	Finally, we use a result in \cite{roberts1997geometric} to further translate the convergence bound in total variation distance to a bound on the $L^2$ norm of the chain's Markov operator.

	Throughout this subsection, fix $\Lambda \subset [n]$ such that $|\Lambda| = n-m$, where $m \in \{2,\dots,n\}$, and let $\ty \in \X_{\Lambda}$ be such that $\pi_{\Lambda}(\ty) > 0$.
	Assume that the assumptions of Lemma~\ref{lem:specind} hold.
	In particular, for $i,j \in - \Lambda$ such that $i \neq j$, there is a coupling kernel $K_{i,j}$ and $\phi_{i,j} < \infty$ such that
	\begin{equation} \label{ine:contraction-1}
		\int_{\X_j \times \X_j} d_{\Lambda,\ty,j}(x'',x''') \, K_{i,j}((x,x'), \df (x'',x''')) \leq \phi_{i,j} \, d_{\Lambda,\ty,i}(x,x')
	\end{equation}
	for $\pi_{\{i\} \mid \Lambda}(\cdot \mid \ty)$-almost every $x,x' \in \X_i$.
	For $i \in - \Lambda$, let $\phi_{i,i} = 0$, and let $K_{i,i}: \X_i \times \X_i \to \B_i \times \B_i$ be a Markov transition kernel such that
	\[
	K_{i,i}((x,x'), A) = \int_{A'} \pi_{\{i\} \mid \Lambda}(x'' \mid \ty) \, \df x'',
	\]
	where $A' = \{(x'',x''): \; x'' \in A\}$.
	($A'$ is measurable when~$A$ is since the former is the intersection of $A \times A$ and the set of points $(x'',x''')$ such that $d_{\Lambda,\ty,i}(x'',x''') = 0$.)
	Then~\eqref{ine:contraction-1} holds even when $i = j$.
	Moreover, the influence matrix $\Phi(\Lambda,\ty)$ can be written as $(\phi_{i,j})$.

	\subsubsection{An altered random walk}
	
	It is convenient to consider a random walk chain that is a slight alteration of Algorithm~\ref{alg:randomwalk}.
	Just like Algorithm~\ref{alg:randomwalk}, Algorithm~\ref{alg:altered} defines a Markov chain that is reversible with respect to a distribution of the form
	\[
	\varphi(\{i\} \times A) = \frac{1}{m} \varpi_i(A), \quad i \in - \Lambda, \; A \in \B_i,
	\]
	with $\varpi_i$ being the measure given by $\pi_{\{i\} \mid \Lambda}(\cdot \mid \ty)$.
	
	\begin{algorithm}
		\caption{One step of an altered random walk associated with $\Lambda \subset [n]$ and $\ty \in \X_{\Lambda}$:}\label{alg:altered}
		\begin{algorithmic}
			\State \textbf{Input:} Current state $(j,x) \in \bigcup_{i \in - \Lambda} (\{i\} \times \X_i)$.
			\State Let $\tx \in \X$ be such that $\tx_{\Lambda} = \ty$ and $\tx_{\{j\}} = x$.
			\State Randomly and uniformly choose a coordinate $j' \in - \Lambda$.
			
			\If{$j'=j$}
			\State Draw $x' \in \X_{\{j\}}$ from $\pi_{\{j\} \mid \Lambda } (\cdot \mid \tx_{\Lambda})$.

			\Else
			
			\State Draw $x' \in \X_{\{j'\}}$ from $\pi_{\{j'\} \mid \Lambda \cup \{j\}} (\cdot \mid \tx_{\Lambda \cup \{j\}})$.
			
			\EndIf
			
			\State \textbf{Return:} New State $(j',x')$.
			
		\end{algorithmic}
	\end{algorithm}
	
	Let $\RW(\cdot,\cdot)$ be the transition kernel for Algorithm~\ref{alg:randomwalk}, and $\RWNL(\cdot,\cdot)$, that for Algorithm~\ref{alg:altered}.
	Each kernel defines a self-adjoint operator on $L_0^2(\varphi)$.
	Indeed, $\RW$ is given by~\eqref{eq:RWf}, i.e.,
	for $f \in L_0^2(\varphi)$, $i \in \Lambda$, and $x \in \X_i$,
	\[
	\RW f(i,x) = \frac{1}{m} \sum_{j \in - (\Lambda \cup \{i\})} \E [ T_jf(X_j) \mid X_i = x, \; \X_{\Lambda} = \ty ] + \frac{f(i,x)}{m} ,
	\]
	where $T_if(x) = f(i,x)$.
	On the other hand,
	\[
		\RWNL f(i,x) = \frac{1}{m} \sum_{j \in - (\Lambda \cup \{i\})} \E [ T_jf(X_j) \mid X_i = x, \; \X_{\Lambda} = \ty ] + \frac{1}{m} \E [ T_if(X_i) \mid \X_{\Lambda} = \ty ].
	\]

	Let $\Range'$ be the space of functions~$f$ in $L_0^2(\varphi)$ such that $T_i f \in L_0^2(\varpi_i)$ for $i \in - \Lambda$.
	$\RW|_{\Range'}$ and $\RWNL|_{\Range'}$, the restrictions of~$\RW$ and~$\RWNL$ to~$\Range'$, are related by the following formula:
	\[
	\RW|_{\Range'} = \RWNL|_{\Range'} + \frac{\Id}{m},
	\]
	where $\Id$ is the identity on~$\Range'$.
	It follows that
	\[
	g(\Lambda,\ty) = 1 - \sup_{\stackrel{f \in \Range'}{f \neq 0}} \frac{\langle f, \RW f \rangle_{\varphi}}{\|f\|_{\varphi}^2} = \frac{m-1}{m} - \sup_{\stackrel{f \in \Range'}{f \neq 0}} \frac{\langle f, \RWNL f \rangle_{\varphi}}{\|f\|_{\varphi}^2} \geq \frac{m-1}{m} - \|\RWNL\|_{\varphi} ,
	\]
	where the first equality is part of~\eqref{eq:1-g} derived in Section~\ref{ssec:proof-randomwalk}.
	Hence, to prove Lemma~\ref{lem:specind}, it suffices to show that
	\begin{equation} \label{ine:specind}
		\|\RWNL\|_{\varphi} \leq \frac{\eta}{m}.
	\end{equation}

	\subsubsection{Convergence in a Wasserstein divergence}

	According to Lemma~\ref{lem:roberts}, to say that~\eqref{ine:specind} holds is to say that the altered random walk chain converges geometrically in the $L^2$ distance at a rate of $\eta/m$.
	To prove this, we first show that the chain converges geometrically in some Wasserstein divergence.

	Let $D: [\bigcup_{i \in - \Lambda} (\{i\} \times \X_i)] \times [\bigcup_{i \in - \Lambda} (\{i\} \times \X_i)] \to [0,\infty]$ be such that, for $i,j \in - \Lambda$ and $x \in \X_i$, $x' \in \X_j$, 
	\[
	D((i,x),(j,x')) = \begin{cases}
		d_{\Lambda,\ty,i}(x,x') & i = j, \\
		\infty & i \neq j.
	\end{cases}
	\]
	A measure in $L_*^2(\varphi)$ has the form 
	\begin{equation} \label{eq:omega}
	\omega(\{i\} \times A) =  a_i \, \omega_i(A), \quad i \in - \Lambda, \; A \in \B_i,
	\end{equation}
	where, $\sum_{i \in - \Lambda} a_i = 1$, and, for $i \in - \Lambda$, $a_i \geq 0$ and $\omega_i \in L_*^2(\varpi_i)$.
	
	The following lemma implies that the altered random walk chain converges geometrically in the Wasserstein divergence induced by~$D$.

	\begin{lemma} \label{lem:wasserstein}
		Let $\omega \in L_*^2(\varphi)$ be as in~\eqref{eq:omega}.
		Then there exist a pair of random walk chains associated with Algorithm~\ref{alg:altered}, denoted by $(I(t), X(t))_{t=0}^{\infty}$ and $(I'(t), X'(t))_{t=0}^{\infty}$, that satisfy the following properties:
		\begin{enumerate}
			\item[(P1)] $(I(0), X(0)) \sim \omega$, and independently, $(I'(0),X'(0)) \sim \varphi$.
			\item[(P2)] $I(t) = I'(t)$ for $t \geq 1$.
			\item[(P3)] Given $I(t) = i_t \in - \Lambda$, the distribution of $X(t)$ is absolutely continuous with respect to~$\varpi_{i_t}$.
			\item[(P4)] There exists a constant $C_{\omega} < \infty$ such that, for each positive integer~$t$,
			\begin{equation} \nonumber
			\E \left[D \left( (I(t), X(t)), (I'(t), X'(t)) \right) \right] \leq C_{\omega} \|\Psi^{t-1}\|_{\infty},
			\end{equation}
			where $\Psi = \Phi(\Lambda,\ty)/m$, and, for any matrix $A = (a_{i,j})$, $\|A\|_{\infty} = \max_i \sum_j |a_{i,j}|$. 
		\end{enumerate}
	\end{lemma}
	
	\begin{proof}
		Construct the two chains according the following Markovian procedure.
		\begin{enumerate}
			\item Let $(I(0), X(0)) \sim \omega$, and independently, $(I'(0),X'(0)) \sim \varphi$.
			\item Draw $I(1) = I'(1)$ randomly and uniformly from $- \Lambda$.
			Denote the observed values of $I(0)$, $I'(0)$, $X(0)$, $X'(0)$, and $I(1) = I'(1)$ by $i_0$, $i'_0$, $z_0$, $z'_0$, and $i_1$ respectively.
			\item For $i,j \in - \Lambda$ and $x \in \X_i$, let $\bar{\Pi}^{j}_{\Lambda,\ty,i,x} = \Pi^{j}_{\Lambda,\ty,i,x}$ if $i \neq j$, and let $\bar{\Pi}^{j}_{\Lambda,\ty,i,x}$ be the probability measure associated with $\pi_{\{j\} \mid \Lambda}(\cdot \mid \ty)$ if $i = j$.
			Independently, draw $X(1)$ from $\bar{\Pi}^{i_1}_{\Lambda,\ty,i_0,z_0}$, and $X'(1)$ from $\bar{\Pi}^{i_1}_{\Lambda,\ty,i'_0,z'_0}$.
			\item For a positive integer~$t$, given $(I(t), X(t)) = (i_t, z_t)$ and $(I'(t), X'(t)) = (i_t, z'_t)$, draw $(I(t+1), X(t+1), I'(t+1), X'(t+1))$ as follows.
			Randomly and uniformly draw $I(t+1) = I'(t+1)$ from $-\Lambda$, and denote the observed value by $i_{t+1}$.
			Then, draw $(X(t+1), X'(t+1))$ using the coupling kernel $K_{i_t,i_{t+1}}( (z_t,z'_t), \cdot )$.
		\end{enumerate}

		It is easy to see that $(I(t), X(t))_t$ and $(I'(t), X'(t))_t$ are both Markov chains whose transition laws follow Algorithm~\ref{alg:altered}, and that they satisfy (P1) and (P2).
		Let us establish (P3) and (P4).
		Fix a positive integer~$t$.
		Let $i_s \in - \Lambda$ for $s = 0,\dots,t$, and let $i'_0 \in - \Lambda$.
		Given $I(s) = I'(s) = i_s$ for $s = 0,\dots,t$ and $I'(0) = i'_0$, the distribution of $(X(t), X'(t))$, denoted by $\nu_{i_0,\dots,i_t; i'_0}$, is given by the following recursive formula:
		\[
		\begin{aligned}
			\nu_{i_0; i'_0}(\df z_0, \df z'_0) &= \omega_{i_0}(\df z_0) \, \varpi_{i'_0}(\df z'_0), \\
			\nu_{i_0,i_1; i'_0}(\df z_1, \df z'_1) &= \int_{\X_{i_0} \times \X_{i'_0}} \bar{\Pi}^{i_1}_{\Lambda, \ty, i_0, z_0}(\df z_1) \, \bar{\Pi}^{i_1}_{\Lambda,\ty,i'_0,z'_0}(\df z'_1) \, \nu_{i_0; i'_0}(\df z_0, \df z'_0), \\
			\nu_{i_0,\dots,i_{s+1};i'_0}(\df z_{s+1}, \df z'_{s+1}) &= \int_{\X_{i_s}^2}  K_{i_s, i_{s+1}} ((z_s, z'_s), (\df z_{s+1}, \df z'_{s+1})) \, \nu_{i_0,\dots,i_s;i'_0}(\df z_s, \df z'_s) ,
		\end{aligned}
		\]
		where $s \geq 1$.
		One can check that, for $s \geq 0$, the distribution given by $A \mapsto \nu_{i_0,\dots,i_s;i'_0}(A \times \X_{i_s})$ is absolutely continuous with respect to $\varpi_{i_s}$, while that given by $A \mapsto \nu_{i_0,\dots,i_s;i'_0}(\X_{i_s} \times A)$ is $\varpi_{i_s}$ itself.
		This implies that (P3) holds.
		Moreover, for $s \geq 1$ and $\nu_{i_0,\dots,i_s;i'_0}$-almost every $(z_s,z'_s) \in \X_{i_s}^2$,
		\[
		\int_{\X_{i_{s+1}}^2} d_{\Lambda,\ty,i_{s+1}}(z_{s+1}, z'_{s+1}) \, K_{i_s, i_{s+1}}((z_s, z'_s), \df (z_{s+1},z'_{s+1})) \leq \phi_{i_s, i_{s+1}} \, d_{\Lambda,\ty,i_s}(z_s,z'_s).
		\]
		Thus, 
		\[
		\begin{aligned}
			& \E \left[D \left( (I(t), X(t)), (I'(t), X'(t)) \right) \mid I(s) = I'(s) = i_s \text{ for } s = 0,\dots,t; \; I'(0) = i'_0 \right] \\
			=& \int_{\X_{i_t}^2} d_{\Lambda,\ty,i_t}(z_t, z'_t) \, \nu_{i_0,\dots,i_t;i'_0}(\df z_t, \df z'_t) \\
			=& \int_{\X_{i_{t-1}}^2} \int_{\X_{i_t}^2} d_{\Lambda,\ty,i_t}(z_t, z'_t) \, K_{i_{t-1},i_t}((z_{t-1},z'_{t-1}), (\df z_t, \df z'_t)) \, \nu_{i_0,\dots,i_{t-1};i'_0}(\df z_{t-1}, \df z'_{t-1}) \\
			\leq & \phi_{i_{t-1}, i_t} \int_{\X_{i_{t-1}}^2} d_{\Lambda,\ty,i_{t-1}}(z_{t-1}, z'_{t-1}) \, \nu_{i_0,\dots,i_{t-1}'i'_0}(\df z_{t-1}, \df z'_{t-1}) \\
			\leq & \left( \prod_{s=1}^{t-1} \phi_{i_s,i_{s+1}} \right) \int_{\X_{i_1}^2} d_{\Lambda,\ty,i_1}(z_1,z'_1) \, \nu_{i_0,i_1;i'_0}(\df z_1, \df z'_1 ).
		\end{aligned}
		\]
		(If $t=1$, then $\prod_{s=1}^{t-1} \phi_{i_s,i_{s+1}}$ is interpreted as~1.)
		Then
		\[
		\begin{aligned}
			& \E \left[D_{\Lambda,\ty} \left( (I(t), X(t)), (I'(t), X'(t)) \right) \right] \\
			\leq & \frac{1}{m^{t+1}} \left( \sum_{i_2,\dots,i_t \in - \Lambda} \prod_{s=1}^{t-1} \phi_{i_s,i_{s+1}} \right) \sum_{i_0,i'_0,i_1 \in - \Lambda} a_{i_0} \int_{\X_{i_1}^2} d_{\Lambda,\ty,i_1}(z_1,z'_1) \, \nu_{i_0,i_1;i'_0}(\df z_1, \df z'_1 ) \\
			 \leq& C_{\omega} \|\Psi^{t-1}\|_{\infty} ,
		\end{aligned}
		\]
		where
		\[
		\begin{aligned}
				C_{\omega} 
			=& \frac{1}{m^2} \sum_{i_0,i'_0,i_1 \in - \Lambda} a_{i_0} \int_{\X_{i_1}^2} d_{\Lambda,\ty,i_1}(z_1,z'_1) \, \nu_{i_0,i_1;i'_0}(\df z_1, \df z'_1 ) \\
			=& \frac{1}{m^2} \sum_{i_0,i_1 \in - \Lambda} a_{i_0} \int_{\X_{i_0}} \int_{\X_{i_1}^2} d_{\Lambda,\ty,i_1}(z_1,z'_1)  \, \bar{\Pi}^{i_1}_{\Lambda,\ty,i_0,z_0}(\df z_1) \, \varpi_{i_1}(\df z'_1) \, \omega_{i_0}(\df z_0) .
		\end{aligned}
		\]
		
		It remains to show that $C_{\omega} < \infty$.
		Fix $i_0, i_1 \in - \Lambda$.
		Recall that (H1) and (H2) in Section~\ref{ssec:specind} are assumed.
		By (H1), 
		\[
		z_1 \mapsto \int_{\X_{i_1}} d_{\Lambda,\ty,i_1}(z_1,z'_1) \, \varpi_{i_1}(\df z'_1)
		\]
		is in $L^2(\varpi_{i_1})$.
		By Cauchy-Schwarz,
		\[
		z_0 \mapsto f(z_0) = \int_{\X_{i_1}} \int_{\X_{i_1}} d_{\Lambda,\ty,i_1}(z_1,z'_1) \, \varpi_{i_1}(\df z'_1) \, \bar{\Pi}^{i_1}_{\Lambda,\ty,i_0,z_0}(\df z_1)
		\]
		is in $L^2(\varpi_{i_0})$.
		Thus,
		\[
		\begin{aligned}
			&\int_{\X_{i_0}} \int_{\X_{i_1}^2} d_{\Lambda,\ty,i_1}(z_1,z'_1)  \, \bar{\Pi}^{i_1}_{\Lambda,\ty,i_0,z_0}(\df z_1) \, \varpi_{i_1}(\df z'_1) \, \omega_{i_0}(\df z_0) \\
			= & \int_{\X_{i_0}} f(z_0) \, \frac{\df \omega_{i_0}}{\df \varpi_{i_0}}(z_0) \, \varpi_{i_0}(\df z_0) \\
			< & \infty.
		\end{aligned}
		\]
		This concludes the proof.
	\end{proof}

	\subsubsection{Convergence in total variation}
	
	To continue, we use the one-shot coupling technique to show that the altered random walk chain converges in total variation distance.
	To be specific, we show the following.
	\begin{lemma} \label{lem:tv}
		Let $\omega \in L_*^2(\varphi)$.
		For $t \geq 0$, denote by $\omega \tilde{R}^t$ the distribution of the $t$th element of a Markov chain associated with Algorithm~\ref{alg:altered}, assuming that the chain's starting distribution (i.e., distribution of its zeroth element) is~$\omega$.
		Then, there exists a constant $C_{\omega} < \infty$ such that, for $t \geq 2$,
		\[
		d_{\scriptsize\mbox{TV}}(\omega \tilde{R}^t, \varphi) \leq C_{\omega} \|\Psi^{t-2}\|_{\infty},
		\]
		where $\Psi = \Phi(\Lambda,\ty)/m$, and $\|\cdot\|_{\infty}$ is defined in Lemma~\ref{lem:wasserstein}.
	\end{lemma}
	\begin{proof}
		Let $(I(t),X(t))_{t=0}^{\infty}$ and $( I'(t),X'(t))_{t=0}^{\infty}$ be a pair of chains associated with Algorithm~\ref{alg:altered} that satisfy (P1) to (P4) in Lemma~\ref{lem:wasserstein}.
		
		Fix $t \geq 2$.
		Given $(I(t-1), X(t-1)) = (i_{t-1},z_{t-1})$ and $(I'(t-1), X'(t-1)) = (i_{t-1}, z'_{t-1})$, proceed as follows.
		Draw~$J$ randomly and uniformly from $- \Lambda$, and call the observed value~$j$.
		If $j = i_{t-1}$, draw~$Z$ from $\pi_{\{j\} \mid \Lambda}(\cdot \mid \ty)$, and let $Z' = Z$.
		If $j \neq i_{t-1}$, do the following:
		Let $\tx \in \X$ be such that $\tx_{\{i\}} = z_{t-1}$ and $\tx_{\Lambda} = \ty$, and let $\tx' \in \X$ be such that $\tx'_{\{i\}} = z'_{t-1}$ and $\tx_{\Lambda} = \ty$.
		Let
		\[
		\begin{aligned}
			p :=& \, p(i_{t-1},j,z_{t-1},z'_{t-1}) \\
			=& \int_{\X_j} \min \left\{ \pi_{\{j\} \mid \Lambda \cup \{i_{t-1}\}}(x \mid \tx_{\Lambda \cup \{i\}}), \, \pi_{\{j\} \mid \Lambda \cup \{i_{t-1}\}}(x \mid \tx'_{\Lambda \cup \{i\}}) \right\} \, \df x.
		\end{aligned}
		\]
		Then 
		\[
		1 - p = d_{\scriptsize\mbox{TV}} \left( \Pi^j_{\Lambda,\ty,i_{t-1},z_{t-1}}, \Pi^j_{\Lambda,\ty,i_{t-1},z'_{t-1}} \right).
		\]
		With probability $p$, draw $Z = Z'$ from the density
		\[
		x \mapsto q(x) = \frac{1}{p} \min \left\{ \pi_{\{j\} \mid \Lambda \cup \{i_{t-1}\}}(x \mid \tx_{\Lambda \cup \{i\}}), \, \pi_{\{j\} \mid \Lambda \cup \{i_{t-1}\}}(x \mid \tx'_{\Lambda \cup \{i\}}) \right\}.
		\]
		With probability $1 - p$, draw~$Z$ from the density
		\[
		x \mapsto \frac{\pi_{\{j\} \mid \Lambda \cup \{i_{t-1}\}}(x \mid \tx_{\Lambda \cup \{i\}}) - p \, q(x )}{1 - p},
		\]
		and independently, draw~$Z'$ from the density
		\[
		x \mapsto \frac{\pi_{\{j\} \mid \Lambda \cup \{i_{t-1}\}}(x \mid \tx'_{\Lambda \cup \{i\}}) - p \, q(x )}{1 - p}.
		\]
		Then $(J,Z) \sim \omega \tilde{R}^t$, while $(J,Z') \sim \varphi$.
		Moreover, given $(I(t-1), X(t-1)) = (i_{t-1},z_{t-1})$, $(I'(t-1), X'(t-1)) = (i_{t-1}, z'_{t-1})$, and $J = j$, the probability of the event $Z = Z'$ is precisely~$p$.
		
		By (H2) and (P3) along with (P2), there is a constant $k < \infty$ such that, almost surely,
		\[
		\begin{aligned}
			p(I(t-1),J,X(t-1),X'(t-1)) &= 1 - d_{\scriptsize\mbox{TV}} \left( \Pi^J_{\Lambda,\ty,I(t-1),X(t-1)}, \Pi^J_{\Lambda,\ty,I(t-1),X'(t-1)} \right) \\
			& \geq 1 - k d_{\Lambda,\ty,I(t-1)}(X(t-1), X'(t-1)) \\
			&= 1 - k D \left((I(t-1),X(t-1)), (I'(t-1), X'(t-1)) \right).
		\end{aligned}
		\]
		It then follows from (P4) that there exists a constant $C'_{\omega}$ unrelated to~$t$ such that
		\[
		\begin{aligned}
			\Prob((J,Z) \neq (J,Z')) &= 1 - \E \left[ p(I(t-1),J,X(t-1),X'(t-1)) \right] \\
			& \leq k \E \left[  D \left((I(t-1),X(t-1)), (I'(t-1), X'(t-1)) \right) \right] \\
			& \leq k C'_{\omega} \|\Psi^{t-2}\|_{\infty}.
		\end{aligned} 
		\]
		By the well-known coupling inequality,
		\[
		d_{\scriptsize\mbox{TV}}(\omega \tilde{R}^t, \varphi) \leq \Prob((J,Z) \neq (J,Z')) \leq k C'_{\omega} \|\Psi^{t-2}\|_{\infty}.
		\]
	\end{proof}
	
	\subsubsection{Convergence in the $L^2$ distance}
	
	To establish~\eqref{ine:specind} and thus Lemma~\ref{lem:specind}, we use Lemma~\ref{lem:tv} to derive a convergence bound in the $L^2$ distance.
	
	Recall that it is assumed that the spectral radius of $\Phi(\Lambda,\ty)$ is no greater than $\eta \in [0,m-1)$.
	Then the spectral radius of $\Psi = \Phi(\Lambda,\ty)/m$ is no greater than $\eta/m \in [0,(m-1)/m)$.
	Since $\|\cdot\|_{\infty}$, as given in Lemma~\ref{lem:wasserstein} is a matrix norm, by Gelfand's formula,
	\[
	\lim_{t \to \infty} \|\Psi^{t-2}\|_{\infty}^{1/t} \leq \frac{\eta}{m},
	\]
	This implies that, for $\rho > \eta/m$, one can find a constant $c_{\rho}$ such that
	\[
	\|\Psi^{t-2}\|_{\infty} \leq c_{\rho} \rho^t
	\]
	for $t \geq 2$.
	
	Fix $\rho \in (\eta/m,1)$.
	For $\omega \in L_*^2(\varphi)$ and $t \geq 0$, let $\omega \tilde{R}^t$ be as defined in Lemma~\ref{lem:tv}.
	Then the said lemma implies that there is a constant $C_{\omega} < \infty$ such that
	\[
	d_{\scriptsize\mbox{TV}}(\omega \tilde{R}^t, \varphi) \leq C_{\omega} \rho^t
	\]
	for $t \geq 0$.
	By Theorem 2.1 in \cite{roberts1997geometric}, the $L^2$ distance between $\omega \tilde{R}^t$ and~$\varphi$ also decreases at a rate of $\rho^t$ or faster;
	moreover, $\|\tilde{R}\|_{\varphi} \leq \rho$.
	Since $\rho \in (\rho/m,1)$ is arbitrary,~\eqref{ine:specind} holds.

	\bigskip
	\noindent{\bf Acknowledgment.} Qian Qin and Guanyang Wang gratefully acknowledge support by the National Science Foundation through grants DMS-2112887 and DMS-2210849.
	The authors thank James P. Hobert for helpful comments.

\bibliographystyle{ims}
\bibliography{qinbib}

\begin{thebibliography}{29}
\expandafter\ifx\csname natexlab\endcsname\relax\def\natexlab#1{#1}\fi
\expandafter\ifx\csname url\endcsname\relax
  \def\url#1{\texttt{#1}}\fi
\expandafter\ifx\csname urlprefix\endcsname\relax\def\urlprefix{URL }\fi

\bibitem[{Alev and Lau(2020)}]{alev2020improved}
\textsc{Alev, V.~L.} and \textsc{Lau, L.~C.} (2020).
\newblock Improved analysis of higher order random walks and applications.
\newblock In \textit{Proceedings of the 52nd Annual ACM SIGACT Symposium on
  Theory of Computing}.

\bibitem[{Anari et~al.(2021{\natexlab{a}})Anari, Jain, Koehler, Pham and
  Vuong}]{anari2021entropic}
\textsc{Anari, N.}, \textsc{Jain, V.}, \textsc{Koehler, F.}, \textsc{Pham,
  H.~T.} and \textsc{Vuong, T.-D.} (2021{\natexlab{a}}).
\newblock Entropic independence ii: optimal sampling and concentration via
  restricted modified log-{S}obolev inequalities.
\newblock \textit{arXiv preprint arXiv:2111.03247} .

\bibitem[{Anari et~al.(2022)Anari, Jain, Koehler, Pham and
  Vuong}]{anari2022entropic}
\textsc{Anari, N.}, \textsc{Jain, V.}, \textsc{Koehler, F.}, \textsc{Pham,
  H.~T.} and \textsc{Vuong, T.-D.} (2022).
\newblock Entropic independence: optimal mixing of down-up random walks.
\newblock In \textit{Proceedings of the 54th Annual ACM SIGACT Symposium on
  Theory of Computing}.

\bibitem[{Anari et~al.(2021{\natexlab{b}})Anari, Liu and
  Gharan}]{anari2021spectral}
\textsc{Anari, N.}, \textsc{Liu, K.} and \textsc{Gharan, S.~O.}
  (2021{\natexlab{b}}).
\newblock Spectral independence in high-dimensional expanders and applications
  to the hardcore model.
\newblock \textit{SIAM Journal on Computing}  FOCS20--1.

\bibitem[{Bj{\o}rstad and Mandel(1991)}]{bjorstad1991spectra}
\textsc{Bj{\o}rstad, P.~E.} and \textsc{Mandel, J.} (1991).
\newblock On the spectra of sums of orthogonal projections with applications to
  parallel computing.
\newblock \textit{BIT Numerical Mathematics} \textbf{31} 76--88.

\bibitem[{Blanca et~al.(2022)Blanca, Caputo, Chen, Parisi,
  {\v{S}}tefankovi{\v{c}} and Vigoda}]{blanca2022mixing}
\textsc{Blanca, A.}, \textsc{Caputo, P.}, \textsc{Chen, Z.}, \textsc{Parisi,
  D.}, \textsc{{\v{S}}tefankovi{\v{c}}, D.} and \textsc{Vigoda, E.} (2022).
\newblock On {Mixing of Markov Chains: Coupling, Spectral Independence, and
  Entropy Factorization}.
\newblock In \textit{Proceedings of the 2022 Annual ACM-SIAM Symposium on
  Discrete Algorithms (SODA)}. SIAM.

\bibitem[{Carlen et~al.(2003)Carlen, Carvalho and
  Loss}]{carlen2003determination}
\textsc{Carlen, E.~A.}, \textsc{Carvalho, M.~C.} and \textsc{Loss, M.} (2003).
\newblock Determination of the spectral gap for {K}ac's master equation and
  related stochastic evolution.
\newblock \textit{Acta mathematica} \textbf{191} 1--54.

\bibitem[{Chen et~al.(2022{\natexlab{a}})Chen, Feng, Yin and
  Zhang}]{chen2022rapid}
\textsc{Chen, X.}, \textsc{Feng, W.}, \textsc{Yin, Y.} and \textsc{Zhang, X.}
  (2022{\natexlab{a}}).
\newblock Rapid mixing of {G}lauber dynamics via spectral independence for all
  degrees.
\newblock In \textit{2021 IEEE 62nd Annual Symposium on Foundations of Computer
  Science (FOCS)}. IEEE.

\bibitem[{Chen and Eldan(2022)}]{chen2022localization}
\textsc{Chen, Y.} and \textsc{Eldan, R.} (2022).
\newblock Localization schemes: A framework for proving mixing bounds for
  markov chains.
\newblock \textit{arXiv preprint arXiv:2203.04163} .

\bibitem[{Chen et~al.(2021)Chen, Galanis, {\v{S}}tefankovi{\v{c}} and
  Vigoda}]{chen2021rapid}
\textsc{Chen, Z.}, \textsc{Galanis, A.}, \textsc{{\v{S}}tefankovi{\v{c}}, D.}
  and \textsc{Vigoda, E.} (2021).
\newblock Rapid mixing for colorings via spectral independence.
\newblock In \textit{Proceedings of the 2021 ACM-SIAM Symposium on Discrete
  Algorithms (SODA)}. SIAM.

\bibitem[{Chen et~al.(2022{\natexlab{b}})Chen, Liu and
  Vigoda}]{chen2022spectral}
\textsc{Chen, Z.}, \textsc{Liu, K.} and \textsc{Vigoda, E.}
  (2022{\natexlab{b}}).
\newblock Spectral independence via stability and applications to holant-type
  problems.
\newblock In \textit{2021 IEEE 62nd Annual Symposium on Foundations of Computer
  Science (FOCS)}. IEEE.

\bibitem[{Diaconis et~al.(2008)Diaconis, Khare and
  Saloff-Coste}]{diaconis2008gibbs}
\textsc{Diaconis, P.}, \textsc{Khare, K.} and \textsc{Saloff-Coste, L.} (2008).
\newblock {G}ibbs sampling, exponential families and orthogonal polynomials
  (with discussion).
\newblock \textit{Statistical Science} \textbf{23} 151--200.

\bibitem[{Feng et~al.(2021)Feng, Guo, Yin and Zhang}]{feng2021rapid}
\textsc{Feng, W.}, \textsc{Guo, H.}, \textsc{Yin, Y.} and \textsc{Zhang, C.}
  (2021).
\newblock Rapid mixing from spectral independence beyond the boolean domain.
\newblock In \textit{Proceedings of the 2021 ACM-SIAM Symposium on Discrete
  Algorithms (SODA)}. SIAM.

\bibitem[{Gerencs{\'e}r(2019)}]{gerencser2019mixing}
\textsc{Gerencs{\'e}r, B.} (2019).
\newblock Mixing time of an unaligned {G}ibbs sampler on the square.
\newblock \textit{Stochastic Processes and their Applications} \textbf{129}
  3570--3584.

\bibitem[{Gerencs{\'e}r and Ottolini(2020)}]{gerencser2020rates}
\textsc{Gerencs{\'e}r, B.} and \textsc{Ottolini, A.} (2020).
\newblock Rates of convergence for {G}ibbs sampling in the analysis of almost
  exchangeable data.
\newblock \textit{arXiv preprint arXiv:2010.15539} .

\bibitem[{Hairer et~al.(2011)Hairer, Mattingly and
  Scheutzow}]{hairer2011asymptotic}
\textsc{Hairer, M.}, \textsc{Mattingly, J.~C.} and \textsc{Scheutzow, M.}
  (2011).
\newblock Asymptotic coupling and a general form of {H}arris’ theorem with
  applications to stochastic delay equations.
\newblock \textit{Probability Theory and Related Fields} \textbf{149} 223--259.

\bibitem[{Jain et~al.(2021)Jain, Pham and Vuong}]{jain2021spectral}
\textsc{Jain, V.}, \textsc{Pham, H.~T.} and \textsc{Vuong, T.~D.} (2021).
\newblock Spectral independence, coupling with the stationary distribution, and
  the spectral gap of the glauber dynamics.
\newblock \textit{arXiv preprint arXiv:2105.01201} .

\bibitem[{Janvresse(2001)}]{janvresse2001spectral}
\textsc{Janvresse, E.} (2001).
\newblock {Spectral gap for Kac's model of Boltzmann equation}.
\newblock \textit{Annals of Probability} \textbf{29} 288--304.

\bibitem[{Johnson and Jones(2015)}]{johnson2015geometric}
\textsc{Johnson, A.~A.} and \textsc{Jones, G.~L.} (2015).
\newblock Geometric ergodicity of random scan {G}ibbs samplers for hierarchical
  one-way random effects models.
\newblock \textit{Journal of Multivariate Analysis} \textbf{140} 325--342.

\bibitem[{Liu et~al.(1995)Liu, Wong and Kong}]{liu1995covariance}
\textsc{Liu, J.~S.}, \textsc{Wong, W.~H.} and \textsc{Kong, A.} (1995).
\newblock Covariance structure and convergence rate of the {G}ibbs sampler with
  various scans.
\newblock \textit{Journal of the Royal Statistical Society, Series B}
  \textbf{57} 157--169.

\bibitem[{Madras and Sezer(2010)}]{madras2010quantitative}
\textsc{Madras, N.} and \textsc{Sezer, D.} (2010).
\newblock Quantitative bounds for {M}arkov chain convergence: {W}asserstein and
  total variation distances.
\newblock \textit{Bernoulli} \textbf{16} 882--908.

\bibitem[{Pillai and Smith(2017)}]{pillai2017kac}
\textsc{Pillai, N.~S.} and \textsc{Smith, A.} (2017).
\newblock Kac’s walk on $n$-sphere mixes in $n\log n$ steps.
\newblock \textit{Annals of Applied Probability} \textbf{27} 631--650.

\bibitem[{Pillai and Smith(2018)}]{pillai2018mixing}
\textsc{Pillai, N.~S.} and \textsc{Smith, A.} (2018).
\newblock On the mixing time of {K}ac’s walk and other high-dimensional
  {G}ibbs samplers with constraints.
\newblock \textit{Annals of Probability} \textbf{46} 2345--2399.

\bibitem[{Qin and Hobert(2022)}]{qin2020wasserstein}
\textsc{Qin, Q.} and \textsc{Hobert, J.~P.} (2022).
\newblock {Wasserstein-based methods for convergence complexity analysis of
  MCMC with applications}.
\newblock \textit{Annals of Applied Probability} \textbf{32} 124--166.

\bibitem[{Roberts and Rosenthal(1997)}]{roberts1997geometric}
\textsc{Roberts, G.~O.} and \textsc{Rosenthal, J.~S.} (1997).
\newblock Geometric ergodicity and hybrid {M}arkov chains.
\newblock \textit{Electronic Communications in Probability} \textbf{2} 13--25.

\bibitem[{Roberts and Rosenthal(2002)}]{roberts2002one}
\textsc{Roberts, G.~O.} and \textsc{Rosenthal, J.~S.} (2002).
\newblock One-shot coupling for certain stochastic recursive sequences.
\newblock \textit{Stochastic processes and their applications} \textbf{99}
  195--208.

\bibitem[{Roberts and Sahu(1997)}]{roberts1997updating}
\textsc{Roberts, G.~O.} and \textsc{Sahu, S.~K.} (1997).
\newblock Updating schemes, correlation structure, blocking and
  parameterization for the {G}ibbs sampler.
\newblock \textit{Journal of the Royal Statistical Society: Series B}
  \textbf{59} 291--317.

\bibitem[{Smith(2014)}]{smith2014gibbs}
\textsc{Smith, A.} (2014).
\newblock A {G}ibbs sampler on the $ n $-simplex.
\newblock \textit{Annals of Applied Probability} \textbf{24} 114--130.

\bibitem[{Wang and Wu(2014)}]{wang2014convergence}
\textsc{Wang, N.-Y.} and \textsc{Wu, L.} (2014).
\newblock Convergence rate and concentration inequalities for {G}ibbs sampling
  in high dimension.
\newblock \textit{Bernoulli} \textbf{20} 1698--1716.

\end{thebibliography}

\end{document}